\documentclass[11pt,oneside]{amsart} 

\usepackage[english]{babel}
\usepackage[T1]{fontenc}
\usepackage{lmodern}

\usepackage{fancyhdr}
\usepackage[usenames]{color}
\usepackage{bbm}
\usepackage[noadjust]{cite}
\usepackage{amsmath, amssymb, latexsym, amsfonts, amscd, amsthm, mathtools}
\usepackage{enumerate}
\usepackage{mathrsfs}

\usepackage{graphicx}

\usepackage[mathscr]{euscript} 

\usepackage{tikz-cd}
\usetikzlibrary{arrows}

\usepackage{hyperref}
\hypersetup{pdftoolbar=true, pdftitle={Template}, pdffitwindow=true, colorlinks=true, citecolor=blue, filecolor=black, linkcolor=red, urlcolor=red, hypertexnames=false}

\newtheoremstyle{thmlike}% name
{8pt}% Space above
{3pt}% Space below
{\slshape}% Body font
{}% Indent amount
{\bfseries}% Theorem head font
{.}% Punctuation after theorem head
{1em}% Space after theorem head
{}% Theorem head spec (can be left empty, meaning `normal')

\newtheoremstyle{deflike}% name
{8pt}% Space above
{3pt}% Space below
{}% Body font
{}% Indent amount
{\bfseries}% Theorem head font
{.}% Punctuation after theorem head
{1em}% Space after theorem head
{}% Theorem head spec (can be left empty, meaning `normal')

\theoremstyle{thmlike}
\newtheorem{theorem}{Theorem}[section]
\newtheorem{proposition}[theorem]{Proposition}
\newtheorem{lemma}[theorem]{Lemma}
\newtheorem{corollary}[theorem]{Corollary}

\theoremstyle{deflike}
\newtheorem{definition}[theorem]{Definition}
\newtheorem{remark}[theorem]{Remark}
\newtheorem{example}[theorem]{Example}
\newtheorem{notation}[theorem]{Notation}

\newtheorem{manualtheoreminner}{Theorem}
\newenvironment{manualtheorem}[1]{%
  \renewcommand\themanualtheoreminner{#1}%
  \manualtheoreminner
}{\endmanualtheoreminner}

\newcommand{\ks}{\mathrm{ks}}

\newcommand{\dR}{\mathrm{dR}}
\renewcommand{\ker}{\operatorname{Ker}}

\newcommand{\Alg}{\operatorname{\mathsf{Alg}}}

\newcommand{\Gal}{\operatorname{Gal}}

\newcommand{\Gm}{\operatorname{\mathbb{G}_{m}}}

\newcommand{\GL}{\operatorname{GL}}
\newcommand{\SL}{\operatorname{SL}}
\newcommand{\Sp}{\operatorname{Sp}}

\newcommand{\U}{\operatorname{U}}
\newcommand{\GU}{\operatorname{GU}}

\newcommand{\Mat}{\operatorname{Mat}}

\renewcommand{\AA}{\mathbb{A}}
\newcommand{\CC}{\mathbb{C}}
\renewcommand{\SS}{\mathbb{S}}

\newcommand{\HH}{\mathcal{H}}

\newcommand{\QQ}{\mathbb{Q}}
\newcommand{\RR}{\mathbb{R}}

\newcommand{\UU}{\mathcal{U}}
\newcommand{\ZZ}{\mathbb{Z}}

\newcommand{\std}{\operatorname{std}}

\newcommand{\tr}{\operatorname{tr}}
\newcommand{\tp}[1]{\prescript{t}{}{#1}} 

\renewcommand{\det}{\operatorname{det}}
\newcommand{\Sym}{\operatorname{Sym}}

\newcommand{\at}{\makeatletter @\makeatother}

\newcommand{\auniv}{\mathscr{A}}

\title{Constructing vector-valued automorphic forms on unitary groups}

\author[T. L. Browning]{Thomas L. Browning}
\address{Thomas L. Browning, Department of Mathematics, University of California, Berkeley, Evans Hall, Berkeley, CA 94720-3840, USA}
\email{tb1004913@berkeley.edu}

\author[P. \v{C}oupek]{Pavel \v{C}oupek}
\address{Pavel \v{C}oupek, Department of Mathematics, Michigan State University, Wells Hall, 619 Red Cedar Road, East Lansing, MI 48824, USA}
\email{coupekpa@msu.edu}

\author[E. Eischen]{Ellen Eischen}
\address{Ellen Eischen, Department of Mathematics, University of Oregon, Fenton Hall, Eugene, OR 97403-1222, USA}
\email{eeischen@uoregon.edu}

\author[C. Frechette]{Claire Frechette}
\address{Claire Frechette, Department of Mathematics, Boston College, Maloney Hall, Chestnut Hill, MA 02467-3806, USA}
\email{frechecl@bc.edu}

\author[S. Hong]{Serin Hong}
\address{Serin Hong, Department of Mathematics,
University of Arizona, 617 N Santa Rita Ave, Tucson, AZ 85721, USA}
\email{serinh@arizona.edu}

\author[S. Y. Lee]{Si Ying Lee}
\address{Si Ying Lee, Department of Mathematics, Stanford University, Building 380, Stanford, CA 94305, USA}
\email{leesy1@stanford.edu}

\author[D. Marcil]{David Marcil}
\address{David Marcil, Department of Mathematics, Columbia University, 2990 Broadway, New York, NY 10027, USA}
\email{d.marcil\at columbia.edu}

\thanks{EE's work was partially supported by NSF Grants DMS-2302011 and DMS-1751281. Part of the paper was completed while EE was in residence at the Institute for Advanced Study; this material is based upon work supported by NSF Grant DMS-1926686.  CF's work was partially supported by NSF Grant DMS-2203042. SH's work was partially supported by the Simons Foundation under Grant Number 814268. During the revision process, PC was partially supported by the NSF grant DMS-2337830.}


\begin{document}

\begin{abstract}
We introduce a method for producing vector-valued automorphic forms on unitary groups from scalar-valued ones.  As an application, we construct an explicit example.  Our strategy employs certain differential operators.  It is inspired by work of Cl\'ery and van der Geer in the setting of Siegel modular forms, but it also requires overcoming challenges that do not arise in the Siegel setting.
\end{abstract}

\maketitle

\section{Introduction}

Automorphic forms play a key role in number theory.  Automorphic forms on unitary groups have proved to be particularly valuable, thanks to structures that arise in this setting. Producing explicit examples of automorphic forms on unitary groups remains challenging, though, and there are relatively few such examples in the literature.

In the setting of unitary groups, one must work with not only scalar-valued but also vector-valued automorphic forms.  We introduce a method for constructing vector-valued automorphic forms on unitary groups from scalar-valued ones.  As an application, we construct an explicit example.  Our strategy employs certain differential operators.

Our approach is inspired by the work Cl\'ery and van der Geer carried out for Siegel modular forms, i.e. automorphic forms on symplectic groups \cite{CleryVDGeer}.  Their work extends a strategy of Witt \cite{Witt}.   Unitary groups  bear certain similarities to symplectic groups.  The setting of unitary groups also presents new challenges, though, which we overcome in this paper.  Related to this, the literature has many more explicit examples for Siegel modular forms than for automorphic forms on unitary groups.  This paper achieves three goals:
\begin{enumerate}
\item Extend Cl\'{e}ry and van der Geer's strategy using differential operators \cite{CleryVDGeer} to unitary groups of all signatures (Proposition \ref{prop:FirstDerivative} and Theorem \ref{thm:HigherDerivatives}).\label{item:diffops}
\item Apply this approach in an explicit example (Theorem \ref{thm:exworks}), which does not carry over trivially from the Siegel case and illustrates challenges new to this setting.
\item Provide a coordinate-free, geometric formulation of our construction (Theorem \ref{thm:reformulation}).  
 While unnecessary for our other goals, this is a more intrinsic approach.
\end{enumerate}

\subsection{Summary of main results and relationship with earlier work}
We draw inspiration from the aforementioned methods that \cite{CleryVDGeer, Witt} introduced for Siegel modular forms for $\Sp(2g, \mathbb{Z})$.  Those forms are defined on Siegel upper half-space
$$\mathfrak{H}_g=\{\tau \in \Mat_{g\times g}(\CC)\;|\; \tp{\tau} = \tau \text{ with } \mathrm{Im}(\tau) \text{ positive-definite} \}.$$
Consider the restriction of a Siegel modular form $f$ for $\Sp(2g, \mathbb{Z})$ via the embedding 
\[ \mathfrak{H}_j\times \mathfrak{H}_{g-j} \rightarrow \mathfrak{H}_g \hspace{1cm}\text{ given by } \hspace{1cm} (\tau', \tau'') \mapsto \begin{pmatrix} \tau'&0\\0&\tau''\end{pmatrix}\]
for some $0\leq j<g$. We write an arbitrary element $\tau \in \mathfrak{H}_g$ as
\[ \tau = \begin{pmatrix} \tau' & x \\ \tp{x} & \tau'' \end{pmatrix}.\]
If $f$ vanishes to order $r$ on $\mathfrak{H}_j\times \mathfrak{H}_{g-j}$, then a certain restricted differential form $\mathrm{d}_x^r f|_{\mathfrak{H}_j\times \mathfrak{H}_{g-j}}$ decomposes into tensor products of Siegel modular forms on $\mathfrak{H}_j$ and $\mathfrak{H}_{g-j}$ \cite[Propositions 2.2 and 2.3]{CleryVDGeer} and can be used to produce explicit vector-valued Siegel modular forms from scalar-valued ones \cite[Section 3]{CleryVDGeer}.

How, if at all, does this extend to the setting of automorphic forms on a unitary group $G$?  When $G$ is of signature $(n,n)$, similarities with the case of symplectic groups suggest a starting point.  Working with other signatures is more complicated.  (For example, unlike in the setting of symplectic groups, one must work with Fourier--Jacobi expansions whose coefficients are not constants but rather theta functions.)  Our results are completely general, in the sense that we handle all signatures.  Given unitary groups $U_\alpha$ and $U_\beta$ with an embedding $U_\alpha\times U_\beta\hookrightarrow U$ for $U$ a larger unitary group, we have a corresponding embedding of symmetric spaces $\HH_\alpha\times \HH_\beta\hookrightarrow \HH$ analogous to the embedding of Siegel upper half-spaces above.  In this case,  $\tau\in \HH$ is given by 
\begin{align*}
\tau = \begin{pmatrix} \tau_\alpha & x \\ y & \tau_\beta \end{pmatrix}
\end{align*}
with $\tau_\alpha\in\HH_\alpha$ and $\tau_\beta\in\HH_\beta$.  Our first main result, Theorem \ref{thm:HigherDerivatives}, extends \cite[Propositions 2.2 and 2.3]{CleryVDGeer} to unitary groups of all signatures and is summarized here:

\begin{manualtheorem}{A}[Summary of Theorem \ref{thm:HigherDerivatives}]\label{differential operator intro}
Suppose $f$ is a scalar-valued automorphic form on the unitary group $U$ that vanishes to order $r$ on $\HH_\alpha\times \HH_\beta$.
\begin{enumerate}
    \item Restricted differential forms $\mathrm{d}_x^r f |_{\HH_\alpha \times \HH_\beta}$ and $\mathrm{d}_y^r f|_{\HH_\alpha \times \HH_\beta}$ decompose into sums of tensor products of automorphic forms on the unitary groups $U_\alpha$ and $U_\beta$.\label{item:thmsummary}
    \item If $f$ is a cusp form, so are all automorphic forms appearing in \eqref{item:thmsummary}. 
\end{enumerate}
\end{manualtheorem}

It is straightforward to recover the approach in \cite{CleryVDGeer, Witt} as a special case of the more general construction in this paper.  The geometric reformulation of our operators (Section \ref{sec:ShimuraVar}, which culminates with Theorem \ref{thm:reformulation}), also suggests that this general construction could be extended still further.  This would likely come at the cost, though, of not getting the sort of explicit example we now describe.  

As an application of Theorem \ref{thm:HigherDerivatives}, we produce an explicit example of a vector-valued automorphic form for unitary groups.  Inspiration comes from  \cite[Section 3]{CleryVDGeer}, which produces vector-valued Siegel modular forms from derivatives of the {\em Schottky form}.  The Schottky form is a Siegel modular form $J$ on $\mathfrak{H}_4$, defined either as the Ikeda lift of the discriminant modular form $\Delta$, or as a difference between the theta series attached to the unimodular lattices $E_8 \oplus E_8$ and $D_{16}^+$. Cl\'ery and van der Geer explicitly describe the Siegel modular forms  $\mathrm{d}_x^4 J |_{\mathfrak{H}_2 \times \mathfrak{H}_2}$ and $\mathrm{d}_x^4 J |_{\mathfrak{H}_3 \times \mathfrak{H}_1}$:
\begin{enumerate}
    \item $J$ vanishes to order $4$ on $\mathfrak{H}_2 \times \mathfrak{H}_2$ with $\mathrm{d}_x^4 J |_{\mathfrak{H}_2 \times \mathfrak{H}_2} = \chi_{1} \otimes \chi_{1}$ for some (scalar-valued) cusp form $\chi_1$ on $\mathfrak{H}_2$.  
    \item $J$ vanishes to order $4$ on $\mathfrak{H}_3 \times \mathfrak{H}_1$ with $\mathrm{d}_x^4 J |_{\mathfrak{H}_3 \times \mathfrak{H}_1} = \chi_{2} \otimes \Delta$ for some (vector-valued) cusp form $\chi_2$ on $\mathfrak{H}_3$. 
\end{enumerate}
The scalar form $\chi_1$ generates the space of cusp forms on $\mathfrak{H}_2$ of weight 10.  The vector-valued form $\chi_2$ generates the space of cusp forms on $\mathfrak{H}_3$ for a specified vector weight. 

In this paper, we consider an analogue of the Schottky form, namely the {\em Hermitian Schottky form} that Hentschel and Krieg defined on a unitary group of signature $(4, 4)$ \cite{HermitianSchottky}.  For the moment, to highlight the relationship with $J$, we denote this form by $\tilde{J}$.  The restriction of the form $\tilde{J}$ to $\mathfrak{H}_4$ is $J$.  It is tempting to assume all the results for the Siegel case will carry over to this setting, but that does not quite turn out to be the case.  We obtain Theorem \ref{thm:exworks}, which concerns forms on Hermitian symmetric spaces $\HH_n$ for unitary groups of signature $(n,n)$ and relies on an embedding $\HH_j\times \HH_{n-j}\hookrightarrow \HH_n$ analogous to the embedding of Siegel upper-half spaces above.

\begin{manualtheorem}{B}[Summary of Theorem \ref{thm:exworks}]\label{vector-valued forms from Schottky}
The scalar form $\tilde{J}$ vanishes to order $4$ on $\HH_3 \times \HH_1$. Furthermore, the form $\mathrm{d}_x^4 \tilde{J} |_{\HH_3 \times \HH_1}$, whose weight is specified in Expression \eqref{eq:weight of dx4 F}, can be written as $M \otimes \Delta$ for some (vector-valued) cusp form $M$ on $\HH_3$.
\end{manualtheorem}

\begin{remark}
    The Fourier coefficients of $M \otimes \Delta$ can be explicitly computed, as seen in the proof of Proposition \ref{prop:vanish4}. However, this is only practical for the first few coefficients.
\end{remark}

Key differences between Theorem \ref{thm:exworks} and the corresponding construction for the Schottky form $J$ include: 
\begin{itemize}
\item The form $\tilde{J}$ does not vanish at all on $\HH_2 \times \HH_2$.
\item We do not claim that $M$ generates the entire space of (vector-valued) cusp forms on $\HH_3$ of its weight.
\end{itemize}
These differences reflect some new challenges that arise in the unitary setting. The first point indicates that the order of vanishing for a modular form does not behave well under the natural embedding of the Siegel space into the Hermitian space. The second point is related to the fact that the dimension of the space parametrizing cusp forms of specified weight currently remains unknown in the unitary setting, in contrast to the Siegel setting.  Consequently, in contrast to the frequent reliance on dimension formulas for cusp forms in the Siegel setting in \cite{CleryVDGeer}, our proof of Theorem \ref{thm:exworks} does not use any dimension formulas.  Instead, we rely on the computation of Fourier coefficients for the derivatives of $\tilde{J}$. 

\begin{remark} \label{rmk:comparison of differential operators}
It is natural to ask about the relationship between our differential operators and others.
The condition that $f$ vanishes to order $r$ on $\mathfrak{H}_j\times\mathfrak{H}_{g-j}$ ensures that our operator is the same (up to a scalar multiple, depending on the normalization) as the operator obtained by applying a particular Maass--Shimura operator (that takes holomorphic forms to nearly holomorphic forms of order $r$, as in, e.g. \cite{harris3lemmas, harrismaass, hasv, shimura81}) to $f$ and restricting the resulting nearly holomorphic form to $\mathfrak{H}_j\times\mathfrak{H}_{g-j}$.  It would also be interesting to explore the relationship between our differential operators and other holomorphic differential operators that have been constructed for symplectic and unitary groups, in particular Rankin--Cohen brackets \cite{banRC, dunn2024rankincohentypedifferentialoperators, martinsenadheera, eholzeribukiyama, ibukiyama}.  
\end{remark}

\subsection{Organization of the paper}
Section \ref{sec:diffops} introduces automorphic forms on unitary groups and certain differential operators.  This section then presents our main results about differential operators, Proposition \ref{prop:FirstDerivative} and Theorem \ref{thm:HigherDerivatives}, producing vector-weight forms from scalar-weight ones.  In Section \ref{sec:example}, we apply the operators in an explicit example.  The main result of this section is Theorem \ref{thm:exworks}.  Section \ref{sec:Proofs} presents proofs of the results from Section \ref{sec:diffops}.  Finally, Section \ref{sec:ShimuraVar}
provides a geometric reformulation of the operators.  While this geometric portion is unnecessary for the results earlier in the paper, it is likely to be of interest to those working with Shimura varieties or seeking an intrinsic, geometric understanding of the operators from Section \ref{sec:diffops}.

\subsection*{Acknowledgements} 
We began this project during the 2022 Arizona Winter School, and we would like to thank the organizers of the Winter School for making this collaboration possible. We would also like to thank Sam Mundy, who served as the project assistant during the five-day Winter School, as well as the remaining participants of the project group: Niven Achenjang, Paulina Fust, Trajan Hammonds, Kalyani Kansal, Kimia Kazemian, Yulia Kotelnikova, Luca Marannino, Aleksander Shmakov, and Wojtek Wawr\'{o}w. Eischen would also like to thank Ger van der Geer for answering her questions about \cite{CleryVDGeer} when she was formulating this project.  We are also grateful to the referee for helpful suggestions.

\section{Automorphic forms and differential operators}\label{sec:diffops}

We begin by introducing automorphic forms on unitary groups, and we construct certain differential operators that act on them.  In this section, we state our main results in a direct (coordinate-dependent) way, because this will be best suited to our application concerning an explicit example in the following section. The proofs of the main assertions (Proposition~\ref{prop:FirstDerivative} and Theorem~\ref{thm:HigherDerivatives}) will be postponed to Section~\ref{sec:Proofs}, after an example in Section~\ref{sec:example}.  For a more comprehensive treatment of automorphic forms on unitary groups from the perspectives employed in this paper, the reader might also consult \cite{
EischenAWS}.

\subsection{Complex automorphic forms}\label{subsec:Autoforms} 
Firstly, we specify notation and conventions for holomorphic automorphic forms on unitary groups considered.

Let $K/\QQ$ be an imaginary quadratic number field, and let $c$ be the unique nontrivial element of $\Gal(K/\QQ)$. Given $k \in K$, we set $\overline{k} := c(k)$. Consider a finite-dimensional $K$-vector space $V$ equipped with a non-degenerate Hermitian pairing $\langle \cdot, \cdot \rangle: V \times V \rightarrow K$. 

\begin{definition}
The \emph{unitary group} $\U(V)=\U(V, \langle \cdot , \cdot \rangle)$ is the algebraic group over $\QQ$  whose points are given by 
$$\U(V)(A)=\{g \in \GL(V\otimes_{\QQ}A)\;|\; \langle gv, gw\rangle=\langle v, w\rangle \text{ for all }v, w \in V\otimes_{\QQ}A\}, \;\; A\in \Alg_{\mathbb{Q}}.$$

More generally, the \emph{similitude unitary group}  $\GU(V)=\GU(V, \langle \cdot , \cdot \rangle)$ is given by
$$\GU(V)(A)=\{(g, \nu(g)) \in \GL(V\otimes_{\QQ}A)\times A^{\times}\;|\; \langle gv, gw\rangle=\nu(g)\langle v, w\rangle\}, \;\; A\in \Alg_{\mathbb{Q}}.$$
\end{definition}

The map $\nu:\GU(V) \rightarrow \Gm$ given by $(g, \nu(g)) \mapsto \nu(g)$ is a group homomorphism, with $\ker \nu=\U(V)$. While $\GU(V)$ will play a role later in Section~\ref{sec:ShimuraVar} for algebro-geometric interpretation, for our present purposes it suffices to work with the group $\U(V)$ only.

Denote by $\AA_f$ the ring of finite ad\`{e}les of $\mathbb{Q}$. A \emph{congruence subgroup} $\Gamma \subseteq \U(V)(\mathbb{Q})$ is a subgroup of the form $$\Gamma = \U(V)(\mathbb{Q})\cap \UU$$
for some open compact subgroup $\UU \subseteq \U(V)(\AA_f)$. Equivalently, upon choosing an integral model $\UU(V)$ of $\U(V),$ $\Gamma$ is a subgroup of $\U(V)(\mathbb{Q})$ that contains the principal congruence subgroup 
$$\Gamma(N)=\{g \in \UU(V)(\ZZ)\;|\; g \mapsto \mathrm{Id} \in \UU(V)(\ZZ/N\ZZ)\}$$
for some $N$ as a finite-index subgroup (this property does not depend on a choice of integral model, see e.g. \cite[Section~4]{Milne} for a detailed discussion).

Choosing a suitable basis of $V\otimes_{\QQ}\RR$, the pairing $\langle \cdot , \cdot \rangle$ can be represented by the matrix $$I_{m, n}=\begin{bmatrix} \mathrm{Id}_m & 0 \\ 0& -\mathrm{Id}_n \end{bmatrix}$$
for some pair of integers $(m, n)$. In this case, $ \U(V)(\RR)$ can be identified with the Lie group $$\U(m, n)=\{\gamma \in \GL_{d}(\CC) \;|\; {\tp{\overline{\gamma}}}I_{m,n}\gamma=I_{m, n}\},$$ where $d = m + n$. We call the pair $(m, n)$ the \emph{signature} of $\U(V)$. If $n$ (resp. $m$) is zero, we simply write $\U(m)$ (resp. $\U(n)$).

The group $\U(m, n)$ naturally acts on the \emph{bounded Hermitian space}
$$\HH_{m, n}=\{\tau \in \Mat_{m \times n}(\CC)\;|\;\mathrm{Id}_n-\tp{\overline{\tau}}\tau \text{ is positive-definite}\}$$
via linear fractional transformations, i.e. 
$$\gamma\tau=(A\tau+B)(C\tau+D)^{-1},\;\; \tau \in \HH_{m, n}, \;\;\;\gamma=\begin{bmatrix}A & B \\ C & D\end{bmatrix} \in \U(m, n)\,$$
(where the sizes of the blocks are determined by $A$ being $m \times m$ and $D$ being $n \times n$).

Given $\gamma \in \U(V)(\RR),$ identified with $\U(m, n)$ as above, and $\tau \in \HH_{m, n}$, we define the \emph{automorphy factors} $\lambda_\gamma(\tau) \in \GL_m(\CC)$, $\mu_{\gamma}(\tau) \in \GL_n(\CC)$ as follows:

$$\lambda_{\gamma}(\tau)=\overline{B}({\tp{\tau}}) +\overline{A}, \quad \mu_{\gamma}(\tau)=C \tau +D, \quad \gamma=\begin{bmatrix}A & B \\ C & D\end{bmatrix} \in \U(m, n).$$

With this setup, we employ the following definition of automorphic forms:

\begin{definition}\label{def:Autoforms}
Let $\U=\U(V)$ be a unitary group of signature $(m, n)$, let $\Gamma \subseteq \U(\QQ)$ be a congruence subgroup, and consider a representation $(\rho, W)$ of $\GL_m(\CC)\times \GL_n(\CC)$. An \emph{automorphic form on $\U$ of level $\Gamma$ and weight $\rho$} is a holomorphic map $f: \HH_{m, n} \rightarrow W$ satisfying 
\begin{equation} \label{eqn:automorphy1}
    f(\tau)
        =(f||_{\rho}\gamma)(\tau):=
    \rho\left(
        \lambda_{\gamma}(\tau), \mu_{\gamma}(\tau)
    \right)^{-1}
    f(\gamma \tau), 
        \;\; 
    \tau \in \HH_{m, n}, \;\; \gamma \in \Gamma\,.
\end{equation}
When the signature is $(1, 1)$ and $\U$ is quasi-split over $\QQ$, we additionally require that $f$ is holomorphic at all cusps. Denote by $M_{\rho}(\Gamma)$ the space of all automorphic forms of weight $\rho$ and level $\Gamma$.
\end{definition}

\begin{remark}
We defer the discussion of holomorphicity at cusps to Remark~\ref{rem:HolomorphicityAtCusps}, after introducing a variant of automorphic forms (amounting to a coordinate change) that is more suitable for the description of the condition. For now let us only remark that in all the cases except of quasi--split unitary group over $\QQ$ of signature $(1, 1)$, the analogue of the condition is automatically satisfied by Koecher's principle \cite{KaiWenLanKoecher}. 
\end{remark}

\begin{remark}
    Using the transitive action of $\U(\RR)$ on $\HH_{m,n}$ and a point whose stabilizer is $K_\infty = \U(m)(\RR) \times \U(n)(\RR) \subset \U(\RR)$, one can view such a map $f$ as a $W$-valued function on $\U(\RR)/K_\infty$. Recall that $\Gamma = \U(\QQ) \cap \UU$ for some open compact subgroup $\UU \subseteq \U(\AA_f)$. 
    
    As is well-known, $f$ can be extended adelically to a left-$\U(\QQ)$-invariant $W$-valued function $\varphi$ on $\U(\AA)/K_\infty$ such that $\varphi$ is right-invariant under $\UU$-translation and right-equivariant under $K_\infty$-translation (and the natural $\rho$ action of $K_\infty \subset \GL_m(\CC) \times \GL_n(\CC)$ on $W$). One says that $\varphi$ is \emph{an automorphic form on $\U$ of level $\UU$ and weight $\rho$}.  For more details, see \cite[Section 3.4]{EischenAWS}

    Lastly, using algebraicity of the associated Shimura variety, one can view $\varphi$ as a section of a certain automorphic vector bundle associated to $\rho$. We do not include precise details here to avoid introducing additional notation that is not necessary later. For instance, one technically needs to address the passage from $\U$ to $\GU$ to make this last step precise. For more details, see \cite[Section 2.7]{EHLS}. 

    In Section \ref{sec:ShimuraVar}, we pass to this geometric point of view, assuming familiarity with the connection between these two perspectives, to discuss algebraic properties of the differential operators constructed below.
\end{remark}

\begin{example} \label{ex:def of rho k l}
    In the case of the representation 
    \[
        \Delta_{k,l}
            =
        \det^k \boxtimes \det^l
            : 
        \GL_m(\CC) \times \GL_n(\CC)
            \to 
        \CC \otimes \CC \simeq \CC\,,
    \]
    we denote the space of automorphic forms of weight $\rho$ (and level $\Gamma$) also by $M_{(k, l)}(\Gamma)$, and refer to its elements as automorphic forms \emph{of weight $(k, l)$}. 
    
    Furthermore, we refer to such automorphic forms as \emph{scalar-valued}, and to general automorphic forms as \emph{vector-valued} when emphasising the distinction. 

    We sometimes write $\Delta_{k,l}$ as $\Delta_{(m,n), (k,l)}$ when we wish to emphasize the choice of ranks for both general linear groups involved in the definition.
\end{example}

\subsection{Variant: Hermitian modular forms}\label{subsec:HermitianForms}
Motivated by the example discussed in Section~\ref{sec:example}, we consider the following variant. Suppose $\U=\U(V)$ is of equal signature, i.e. $m=n$, and let us additionally assume that $\U$ is quasi-split over $\QQ$. Then a suitable choice of basis of $V\otimes_{\QQ}\RR$ (in fact, of $V$ by the quasi-splitness assumption) allows one to express the pairing $\langle \cdot, \cdot \rangle $ by the matrix $i\eta_n$ where $$\eta_n=\begin{bmatrix} & -\mathrm{Id}_n \\ \mathrm{Id}_n & \end{bmatrix},$$ identifying $\U(\RR)$ with
$$\U(\eta_n)=\{\gamma \in \GL_{2n}(\CC)\;|\; \tp{\overline{\gamma}}\eta_n \gamma=\eta_n\}.$$
The group $\U(\eta_n)$ then naturally acts on the \emph{unbounded Hermitian space}
$$\HH_n=\{\tau \in \Mat_{n\times n}(\CC)\;|\; i({^{t}\overline{\tau}}-\tau) \text{ is positive-definite} \}$$
again via fractional linear transformations, i.e.
$$\gamma\tau=(A\tau+B)(C\tau+D)^{-1},\;\; \tau \in \HH_n, \;\;\gamma=\begin{bmatrix}A & B \\ C & D\end{bmatrix} \in \U(\eta_{n})\,.$$
We then define, for $\gamma \in \U(\RR)$ viewed as an element of $\U(\eta_n),$ the automorphy factors $\lambda_{\gamma}(\tau), \mu_{\gamma}(\tau) \in \GL_n(\CC)$ as follows:
$$\lambda_{\gamma}(\tau)=\overline{C}({\tp{\tau}}) +\overline{D}, \quad \mu_{\gamma}(\tau)=C \tau +D, \quad \gamma=\begin{bmatrix}A & B \\ C & D\end{bmatrix} \in \U(\eta_n).$$

\begin{definition}\label{def:HermitianForms}
Let $\U=\U(V)$ be a unitary group of signature $(n, n)$ and quasi-split over $\QQ$. Consider a congruence subgroup $\Gamma \subseteq \U(\QQ)$ and a representation $(W, \rho)$ of $\GL_n(\CC)\times \GL_n(\CC)$. A \emph{Hermitian modular form of level $\Gamma$ and weight $\rho$} is a holomorphic map $f: \HH_n \rightarrow W$ satisfying 
$$f(\tau)=(f||_{\rho}\gamma)(\tau):=\rho\left(\lambda_{\gamma}(\tau), \mu_{\gamma}(\tau)\right)^{-1}f(\gamma \tau), \;\; \tau \in \HH_{n}, \;\; \gamma \in \Gamma.$$
Moreover, when the signature is $(1, 1)$, we additionally require that $f$ is holomorphic at every cusp.
\end{definition}
Hermitian modular forms were first introduced by Hel Braun in \cite{braun1, braun2, braun3}.

\begin{remark}\label{rem:HolomorphicityAtCusps}
Let us spell out the meaning of the holomorphicity condition along the lines of \cite[\S~5]{ShimuraArithmeticity}. Given a Hermitian modular function $f$ of level $\Gamma$ and weight $\rho$ (i.e. a function satisfying the modularity condition of Definition~\ref{def:HermitianForms}, but not necessarily the holomorphicity at cusps), $f$ admits a Fourier expansion of the form 
\begin{equation}\label{eqn:Fourier}
f(\tau)=\sum_{h \in M^\vee}\mathbf{c}(h)\exp(2\pi i \tr(h \tau))
\end{equation} 
where
\begin{itemize}
\item{$M$ is a $\ZZ$-lattice of complex Hermitian $n\times n$ matrices $\alpha$ with $\begin{bmatrix} \mathrm{Id}_n & \alpha \\ 0& \mathrm{Id}_n\end{bmatrix} \in\Gamma$, i.e. such that $f(\tau+\alpha)=f(\tau)$ (existence of such lattice is guaranteed by the quasi-splitness assumption),} 
\item{$M^{\vee}$ is the $\ZZ$-lattice of all complex Hermitian $n\times n$ matrices $h$ with $\tr(h \alpha) \in \ZZ$ for all $\alpha \in M$,}
\item{$\mathbf{c}(h)$ are vectors in the underlying vector space $W_{\rho}$  of $\rho$.}
\end{itemize}
When $n=1$, the Hermitian matrices $h \in M^{\vee}$ and the coefficients $c(h)$ are just real and complex numbers, respectively. In this case, we say that \emph{$f$ is holomorphic at $\infty$} if $c(h)=0$ whenever $h<0$. We say that \emph{$f$ is holomorphic at all cusps} if for all $\beta \in \SL_2(\mathbb{Q})$, the function $ f||_{\rho}\beta$ (automorphic of level $\beta^{-1}\Gamma \beta$) is holomorphic at $\infty$.
\end{remark}

We emphasize that for higher $n$, the analogous condition is automatic by Koecher's principle, which in the presence of Fourier coefficients can be stated as follows.

\begin{proposition}[Koecher's principle; {\cite[Proposition~5.7]{ShimuraArithmeticity}}]\label{prop:Koecher}
When $n>1$ and $f$ is a Hermitian modular function of some weight $\rho$ and level $\Gamma \subseteq \U(\eta_n)$, in the expression \eqref{eqn:Fourier} one has $\mathbf{c}(h)\neq 0$ only if $h$ is positive--semidefinite.
\end{proposition}

Fourier expansions also allow us to conveniently define Hermitian cusp forms.

\begin{definition}\label{def:CuspForms}
Given a Hermitian modular form $f$ of level $\Gamma$ and weight $\rho$, $f$ is called a \textit{cusp form} if for any $\beta \in \U(\QQ)$, in the Fourier expansion \eqref{eqn:Fourier} of $f||_{\rho}\beta$, one has $\mathbf{c}(h)=0$ whenever $h$ is not positive-definite.
\end{definition}

We denote the space of all Hermitian modular forms of weight $\rho$ and level $\Gamma$ again by  $M_{\rho}(\Gamma)$ (we hope that there is little potential for substantial confusion). Similarly, we employ the notation $M_{(k, l)}(\Gamma)$ when $\rho=\det^k\boxtimes \det^l$, and refer to Hermitian modular forms of this type as \emph{scalar-valued}. We denote the space of all all Hermitian cusp forms by $S_\rho(\Gamma)$ (and $S_{(k, l)}(\Gamma)$ if $\rho=\det^k\boxtimes \det^l$).

Following Shimura's notation from \cite{ShimuraArithmeticity}, when convenient to make the disctinction, we will refer to unitary groups and automorphic forms in the coordinates described in Section~\ref{subsec:Autoforms} as the ``case (UB)'' (where ``UB'' stands for ``unitary ball'') and to Hermitian forms on $U(\eta_n)$ in the sense of this section as the ``case (UT)'' (i.e., ``unitary tube''). 

\subsection{Further variants}\label{subsec:UTs}
Following \cite{ShimuraUnitaryGrps}, let us mention two further variants that will serve an auxiliary purpose thanks to their convenience in expressing automorphc forms via Fourier expansions. Let us fix a unitary group $\U$ of signature $(m, n)$.

Firstly, we consider the case $m \neq n$. Without loss of generality, let us assume that $m>n$. Then $\U(\RR)$ may also be realized as the group 
$$\widetilde{\U}(m, n)=\{g \in \GL_{n+m}(\CC)\;|\; \tp{\overline{g}}\eta_{m,n}g=\eta_{m, n}\},$$ where
$$\eta_{m,n}=\begin{bmatrix} & & \mathrm{Id}_n \\
 & S & \\
-\mathrm{Id}_n &  &\end{bmatrix}$$
with $S$ diagonal skew-Hermitian and such that $-iS$ is positive-definite. Then, $\U(\RR)$ naturally acts on the symmetric space 
$$\widetilde{\HH}_{m, n}=\left\{z=\begin{bmatrix}\tau \\ u\end{bmatrix}\;|\; \tau \in \Mat_{n \times n}(\CC),  u \in \Mat_{(m-n) \times n}(\CC),\;\;  -i(\tau-\tp{\overline{\tau}})+i\tp{\overline{u}}Su>0\right\}.$$
This is the convention taken in \cite{ShimuraUnitaryGrps} (up to order of coordinates). The appropriate action and automorphy factors $\lambda_{\gamma}(z), \mu_{\gamma}(z)$ (in the notation of \textit{loc. cit.}, $\kappa(\gamma, z)$ and $\mu(\gamma, z)$, resp.) are given as follows. For a matrix 
$$\gamma =\begin{bmatrix} A_1 & B_1 & C_1 \\ A_2 & B_2 & C_2 \\ A_3 & B_3 & C_3  \end{bmatrix} \in \widetilde{\U}(m, n)$$ (with diagonal blocks square of sizes $n, m-n$ and $n$, respectively) and $z=\begin{bmatrix} \tau \\ u \end{bmatrix} \in \widetilde{\HH}_{m, n}$, one has 
$$\gamma\left(z\right)=\left(\begin{bmatrix}A_1 & B_1 \\ A_2 & B_2\end{bmatrix}\begin{bmatrix}\tau \\ u\end{bmatrix}+\begin{bmatrix}C_1 \\ C_2\end{bmatrix}\right)\left(A_3\tau+B_3u+C_3\right)^{-1},$$
$$\lambda_{\gamma}(z)=\begin{bmatrix}\overline{A}_3 \tp{\tau}+\overline{C}_3 & \overline{A}_3 \tp{u}+\overline{B}_3 \overline{S}^{-1} \\ \overline{S}(\overline{A}_2 \tp{\tau}+\overline{C}_2) & \overline{S}(\overline{A}_2\tp{u}+\overline{B}_2 \overline{S}^{-1})\end{bmatrix}, \;\; \mu_{\gamma}(z)=A_3\tau+B_3u+C_3\,.
$$ The corresponding automorphic forms are then defined as in Definition~\ref{def:HermitianForms}. As explained in \textit{loc. cit.}, the automorphic forms admit \textit{Fourier--Jacobi expansions}, that is, they can be written in the form
$$f\left(\tau, u\right)=\sum_{h}\mathbf{c}(u; h)\exp(2 \pi i \tr(h\tau)),$$
with $h$ ranging over a suitable lattice of $n \times n$ Hermitian matrices and the coefficient functions $\mathbf{c}(u; h)$ are certain theta functions. Similarly to the previous case, the Koecher's principle implies that $\mathbf{c}(u; h)=0$ unless $h$ is non-negative (cf. \cite[p. 570]{ShimuraUnitaryGrps}). 

Lastly, let us consider the case $m=n>1$, but when $\U$ is not quasi-split over $\QQ$. By \cite[\S6]{ShimuraUnitaryGrps}, the following form of the group can nonetheless be achieved: 
$$\U(\RR)\simeq \U(\widetilde{\eta_n})=\{g \in \GL_{2n}(\CC)\;|\; \tp{\overline{g}}\widetilde{\eta_{n}}g=\widetilde{\eta_{n}}\},$$
with 
$$\widetilde{\eta_n}=\begin{bmatrix}t &&& \\  &&&\mathrm{Id}_{n-1}\\ & & s & \\ & -\mathrm{Id}_{n-1} && \end{bmatrix},$$ where $s, t \in K$ are pure imaginary elements whose product is positive. The associated symmetric space is given as
$$\widetilde{\HH}_n=\{Z \in \Mat_{n \times n}(\CC)\;|\; i \begin{bmatrix} \tp{\overline{Z}} & \mathrm{Id}_n \end{bmatrix} \widetilde{\eta}_n \begin{bmatrix} Z\\ \mathrm{Id}_{n} \end{bmatrix} >0\}.$$
Automorphic forms on $\widetilde{H}_n$ admit Fourier--Jacobi expansions of the form
$$f(\tau)=\sum_{h}\mathbf{c}(u, v, w; h)\exp(2 \pi i \tr(h\tau')),$$
where we consider the coordinates $$\widetilde{H}_n \ni \tau =\begin{bmatrix}u & v \\ w & \tau'\end{bmatrix},\;\; u \in \CC, \;\tau \in \Mat_{(n-1)\times(n-1)}(\CC)$$
and $h$ ranges through a suitable lattice of $(n-1)\times(n-1)$ Hermitian matrices. Once again, the coefficients $\mathbf{c}(u, v, w; h)$ vanish unless $h$ is non-negative.

\subsection{Restrictions of automorphic forms}\label{subsec:C_Restriction}

Let us fix a choice of $K, V, \langle \cdot, \cdot \rangle,$ the corresponding unitary group $\U=\U(V)$ of signature $(m, n)$, and a choice of a congruence subgroup $\Gamma$ as in the previous section. Consider a decomposition $V=V_1\oplus V_2$ where $V_1, V_2$ are vector subspaces orthogonal for the pairing $\langle -, - \rangle$. This induces a natural embedding
$$\eta:\U_1 \times \U_2 \hookrightarrow \U$$
where $\U_i=\U(V_i),\;\;i=1,2$. Additionally, let us choose congruence subgroups $\Gamma_i\subseteq \U_i(\QQ)$ such that $\eta(\Gamma_1 \times \Gamma_2)\subseteq \Gamma$. Let $(m_i, n_i)$ be the corresponding signatures, so that $(m_1, n_1)+(m_2, n_2)=(m, n)$.

We fix a choice of one of the two versions of coordinates discussed in Sections~\ref{subsec:Autoforms} and \ref{subsec:HermitianForms}, resp., the same for all groups $\U, \U_1 $ and $\U_2$. Explicitly, we consider one of the following two options:

\begin{enumerate}[(1)]
\item{In the case (UB), we identify $\U(\RR)$ with $\U(m, n),$ $\U_1(\RR)$ with $\U(m_1, n_1)$ and $\U_2(\RR)$ with $\U(m_2, n_2)$. In this case, we set $\HH=\HH_{m, n},$ $\HH^{(1)}=\HH_{m_1, n_1}$ and $\HH^{(2)}=\HH_{m_2, n_2}$.}
\item{When $m=n$, $m_1=n_1$ and $m_2=n_2$, with all the groups $\U_1, \U_2, \U$ quasi-split over $\QQ$, we may consider the case (UT), i.e. we identify $\U(\RR)$ with $\U(\eta_n)$ and $\U_i(\RR)$ with $\U(\eta_{n_i})$, $i=1, 2$. In this case, we set $\HH=\HH_{n},$ $\HH^{(1)}=\HH_{n_1}$ and $\HH^{(2)}=\HH_{n_2}$.}
\end{enumerate}

We will discuss both of these cases at once, in a unified way. To that end, from now on we use the term ``automorphic form'' to refer both to automorphic forms in the sense of Definition~\ref{def:Autoforms} in the case (UB), as well as to Hermitian modular forms (from Definition~\ref{def:HermitianForms}) in the case (UT).

Regardless of which case occurs, the bases of $(V_i)_\RR$ giving the chosen coordinates may be chosen so that $\eta$ becomes the map 
\begin{align*}
\eta: \U_1(\RR)\times \U_2(\RR) &\hookrightarrow \U(\RR), \;\;\; \left(\begin{bmatrix}a_1 & b_1 \\ c_1 & d_1 \end{bmatrix}, \begin{bmatrix}a_2 & b_2 \\ c_2 & d_2 \end{bmatrix}\right) \mapsto \begin{bmatrix}a_1 && b_1 & \\ & a_2 & & b_2 \\ c_1 && d_1 & \\ & c_2 & & d_2\end{bmatrix}.  
\end{align*}
The corresponding embedding of symmetric spaces is then given by the map
\begin{align*}
\iota: \HH^{(1)}\times \HH^{(2)} &\hookrightarrow \HH,\;\;\; (\tau_1, \tau_2)\mapsto \begin{bmatrix}\tau_1 & \\ & \tau_2\end{bmatrix},
\end{align*}
which is clearly $\U_1(\RR)\times \U_2(\RR)$-equivariant in the obvious sense. 

We fix a notation for coordinates on $\HH$ compatibly with the embeddings, that is,
\begin{equation} \label{eq:coordinates}
\HH \ni \tau=\begin{bmatrix}\tau_1 & x \\ y & \tau_2 \end{bmatrix}, \;\;\; \tau_1 \in \HH^{(1)}, \; \tau_2\in \HH^{(2)},
\end{equation}
where $x=(x_{ij})$ and $y=(y_{ji})$ are rectangular blocks whose dimensions are determined by the blocks $\tau_1, \tau_2$. If needed, we will refer to the coordinates $(x_{ij})$ as ``$x$-coordinates'', and similarly, to $(y_{ji})$ as ``$y$-coordinates''.

Observe that for $\gamma=\eta(\gamma_1, \gamma_2) \in \Gamma$ with $\gamma_i \in \Gamma_i$ and
 $\tau=\iota(\tau_1, \tau_2)$ with $\tau_i \in \HH^{(i)},$ we have 
$$\lambda_{\gamma}(\tau)=\begin{bmatrix}\lambda_{\gamma_1}(\tau_1)& \\ & \lambda_{\gamma_2}(\tau_2)\end{bmatrix},\;\;\; \mu_{\gamma}(\tau)=\begin{bmatrix}\mu_{\gamma_1}(\tau_1)& \\ & \mu_{\gamma_2}(\tau_2)\end{bmatrix}\,.$$
It follows that, when $k, l$ are arbitrary integers, the restriction of a form $f \in M_{(k,l)}(\Gamma)$ to a function on $\HH^{(1)} \times \HH^{(2)}$ defines a map
\begin{equation} \label{eq:classical restriction}
    M_{(k, l)}(\Gamma) \to M_{(k, l)}(\Gamma_1) \otimes M_{(k, l)}(\Gamma_2)\,.
\end{equation}
In terms of the coordinates \eqref{eq:coordinates}, the map is given by $f\mapsto f\vert_{\substack{x=0\\y=0}},$ i.e. by restriction to the locus where all $x_{ij}=y_{ji}=0$.

\subsection{The construction}\label{subsec:Construction}

Fix the notation for $(k, l), \U_i, \U, \HH^{(i)}, \Gamma_i, \Gamma$ etc. as in Section~\ref{subsec:C_Restriction}. Our next goal is to describe the desired differential operators on automorphic forms in coordinates. To formulate the construction, the following calculus notation will be useful.

\begin{notation}
Given a function $f(x_1, \dots, x_l)$, and an (ordered) tuple of indicies $\alpha=(i_1, i_2, \dots, i_r),$ denote $$\partial_{x_\alpha}f=\frac{\partial^{r}f}{\partial x_{i_1}\partial x_{i_2}\dots \partial x_{i_r}}.$$ 
For $r \geq 1$, $\mathrm{d}^r f$ denotes the $r$-th total differential of $f$, that is, the (symmetric) $r$-linear form on the tangent space given in coordinates as 
$$
    \mathrm{d}^r f
        =
    \sum_{\alpha}\partial_{x_\alpha}f \mathrm{d}x_\alpha\,,
$$
where $\alpha=(i_1, i_2, \dots i_r)$ runs over all (ordered) $r$-tuples of indices for coordinates, and $\mathrm{d}x_\alpha$ denotes $\mathrm{d}x_{i_1}\mathrm{d}x_{i_2} \dots \mathrm{d}x_{i_r}$. 

When $g=(g_1, g_2, \dots g_l)$ is a vector function in variables $y_1, y_2, \dots, y_t$, we denote by $\mathrm{d}^r g$ the tensor $(\mathrm{d}^rg_1, \mathrm{d}^rg_2, \dots \mathrm{d}^r g_l)$ (so that $\mathrm{d} g=\mathrm{d}^1 g$ agrees with the usual meaning of the tangent map). 
\end{notation}

\begin{remark}
In this notation, we have the following convenient forms of the chain rule:
$$\mathrm{d}(f\circ g)=\mathrm{d}f \circ \mathrm{d}g, \;\;$$
$$\mathrm{d}^r(f\circ g)=\sum_{a=1}^r\sum_{b_1+b_2+\dots+b_a=r}\mathrm{d}^af \circ (\mathrm{d}^{b_1}g,\mathrm{d}^{b_2}g, \dots, \mathrm{d}^{b_a}g). \;\;$$
\end{remark}

We are now ready to proceed with the construction. To make the construction more transparent, we start by formulating the case of the first derivative separately. 

For every pair of non-negative integers $(m, n)$, define \begin{align*}
    \rho^+_{(m, n), (k, l)} 
        :=
    \Delta_{(m,n), (k,l)} \otimes (\rho_{\mathrm{std}} \boxtimes \mathbf{1}) 
        : 
    \GL_m(\CC) \times \GL_n(\CC) 
        &\rightarrow 
    \GL(\CC^m\otimes \CC) \simeq \GL_m(\CC), \\ (U, V)
        &\mapsto 
    \det(U)^k \det(V)^l U\,, 
        \\ \\ 
    \rho^-_{(m, n), (k, l)} 
        :=
    \Delta_{(m,n), (k,l)} \otimes (\mathbf{1} \boxtimes \rho_{\mathrm{std}})
        :
    \GL_{m}(\CC) \times \GL_{n}(\CC)
        &\rightarrow 
    \GL(\CC\otimes \CC^n) \simeq \GL_n(\CC), \\ (U, V)
        &\mapsto 
    \det(U)^k \det(V)^l V\,,
\end{align*}
as representations of $\GL_{m}(\CC)\times \GL_{n}(\CC)$. Here, $\Delta_{(m,n), (k,l)}$ is as in Example \ref{ex:def of rho k l} and $\rho_{\mathrm{std}}$ (resp. $\mathbf{1}$) denotes the standard (resp. trivial) representation of the appropriate dimension. 

\begin{remark}
    When $(m,n) = (m_i, n_i)$ for $i$ = 1 or 2, as in Section \ref{subsec:C_Restriction}, we simply write $\rho^\pm_{i, (k,l)}$ for $\rho^\pm_{(m_i, n_i), (k,l)}$.
\end{remark}

For a holomorphic function $f: \HH \rightarrow \CC$, we denote by $\mathrm{d}_{x} f|_{\HH^{(1)}\times \HH^{(2)}}$ the form \[ 
        \mathrm{d}_{x} f|_{\HH^{(1)}\times \HH^{(2)}}=
        \left.
            \mathrm{d}_{x}f
        \right\vert_{
            \substack{x=0\\y=0}
        }
            =
        \left.
            \sum_{i, j} 
                \frac{\partial f}{\partial x_{ij}} \mathrm{d}x_{ij}
        \right\vert_{
            \substack{x_{ij}=0 \\ y_{ji}=0},
        }
    \]
that is, the form obtained from $\mathrm{d} f$ by projection onto the span of differentials of the $x$-coordinates, and then setting all $x$- and $y$-coordinates to $0$. 

Similarly, $\mathrm{d}_{y} f|_{\HH^{(1)}\times \HH^{(2)}}$ is defined as 
        $$\mathrm{d}_{y} f|_{\HH^{(1)}\times \HH^{(2)}}=\left.\mathrm{d}_{y} f\right\vert_{\substack{x=0\\y=0}}=\left.\sum_{j, i} \frac{\partial f}{\partial y_{ji}}\mathrm{d}y_{ji}\right\vert_{\substack{x_{ij}=0 \\ y_{ji}=0}},$$
i.e. the analogous form where we project $\mathrm{d} f$ onto $\mathrm{d}y$-coordinates instead before restricting to $\HH^{(1)}\times\HH^{(2)}$.   

\begin{proposition}\label{prop:FirstDerivative}
Let $f: \HH \rightarrow \CC$ be an automorphic form of weight $(k, l)$, and assume that the restriction $f|_{\HH^{(1)} \times \HH^{(2)}}$ vanishes. 

Then, the differential form 
   $\mathrm{d}_{x} f|_{\HH^{(1)}\times \HH^{(2)}}$ 
is a tensor product of vector-valued automorphic forms of level $\Gamma_1$ and $\Gamma_2$ respectively, and weight $\rho^+_{1, (k,l)}$ and $\rho^-_{2, (k,l)}$ respectively. 

Similarly, the form 
    $\mathrm{d}_{y} f|_{\HH^{(1)}\times \HH^{(2)}}$ 
is a tensor product of vector-valued automorphic forms of level $\Gamma_1$ and $\Gamma_2$ respectively, and weight $\rho^-_{1, (k,l)}$ and $\rho^+_{2, (k,l)}$ respectively.
\end{proposition}

\begin{remark}\label{rem:MatricesAsTensors}
To better understand the content of Proposition~\ref{prop:FirstDerivative}, let us be more explicit about the expected modularity rule. By convention, we identify the space $\Mat_{m_1 \times n_2}(\CC)$ of $m_1\times n_2$ complex matrices with the $\GL_{m_1}(\CC)\times \GL_{n_2}(\CC)$-representation $\rho_{\mathrm{std}}\boxtimes \rho_{\mathrm{std}}$, with the action given by the formula
$$\rho(A, B)(X)=AX \tp{B}\,,\;\; (A, B) \in \GL_{m_1}(\CC)\times \GL_{n_2}(\CC)\,, \;\; X \in \Mat_{m_1 \times n_2}(\CC)\,.$$

The form $\left.\mathrm{d}_xf \right\vert_{\substack{x=0\\y=0}}$ naturally takes values in the dual space $\Mat_{m_1 \times n_2}(\CC)^{\vee},$ on which $\GL_{m_1}(\CC)\times \GL_{n_2}(\CC)$ acts via
$$\rho'(A, B)(\alpha)(X)=\alpha(\tp{A}X B) \;\;(=\alpha(\rho(\tp{A}, \tp{B})(X))\;,$$ and it is easy to see that the resulting action makes  $\Mat_{m_1 \times n_2}(\CC)^{\vee}$ into a representation which is again isomorphic to $\rho_{\mathrm{std}}\boxtimes \rho_{\mathrm{std}}$.
 
The modularity condition of Proposition~\ref{prop:FirstDerivative} can then be rephrased as follows: assuming the vanishing of $f|_{\HH^{(1)} \times \HH^{(2)}}$, for an element $\gamma$ of the form 
\begin{align*}\gamma&=\begin{bmatrix}A & B \\ C & D \end{bmatrix}=\begin{bmatrix}a_1 && b_1 & \\ & a_2 & & b_2 \\ c_1 && d_1 & \\ & c_2 & & d_2\end{bmatrix} =\eta(\gamma_1, \gamma_2) \in \Gamma, \\ 
(\gamma_1, \gamma_2)&=\left(\begin{bmatrix}a_1 & b_1 \\ c_1 & d_1 \end{bmatrix}, \begin{bmatrix}a_2 & b_2 \\ c_2 & d_2 \end{bmatrix}\right) \in \Gamma_1 \times \Gamma_2,
\end{align*} 
and $(\tau_1, \tau_2) \in \HH^{(1)}\times \HH^{(2)},$ we have 
\begin{align*}
\mathrm{d}_{x} f\left(\begin{bmatrix}\tau_1 & \\ & \tau_2\end{bmatrix}\right)=\\
=\rho'(\rho^+_{m_1, n_1}(\lambda_{\gamma_1}(\tau_1) &, \mu_{\gamma_1}(\tau_1))^{-1},{\rho^-_{m_2, n_2}(\lambda_{\gamma_2}(\tau_2), \mu_{\gamma_2}(\tau_2))}^{-1})\left(\mathrm{d}_{x} f\left(\begin{bmatrix}\gamma_1\tau_1 & \\ & \gamma_2\tau_2\end{bmatrix}\right)\right).
\end{align*}

Similar interpretation applies to the case of the form $\mathrm{d}_{y} f|_{\HH^{(1)}\times \HH^{(2)}}$ (where we instead consider the space $\Mat_{m_2 \times n_1}(\CC)$ etc.).
  \end{remark}

Let us now describe the general case of higher order derivatives. Given a holomorphic map $f: \HH \rightarrow \CC$, the form $\mathrm{d}^r f$ projected onto the $\mathrm{d}x$-coordinates and restricted  to $\HH^{(1)} \times \HH^{(2)}$, 
\begin{equation} \label{eq:definition of d x r}
    \left.
        \mathrm{d}^r_x f
    \right|_{\HH^{(1)}\times \HH^{(2)}} =
    \left.
        \mathrm{d}^r_x f
    \right\vert_{
        \substack{x=0\\y=0}
    }   =
    \sum_\alpha
        \left. 
            \partial_{x_\alpha} f
        \right\vert_{
            \substack{x=0\\y=0}
        }
        \mathrm{d}x_{\alpha}
\end{equation}
(where $\alpha$ runs over all $r$-tuples of indices for the $x$-coordinates), is a map from $\HH^{(1)} \times \HH^{(2)}$ that naturally lands in the space of $r$-linear forms on $m_1 \times n_2$ complex matrices, that is, in $\Sym^{r}(\Mat_{m_1 \times n_2}(\CC)^{\vee})$.

The natural left action of $\GL_{m_1}(\CC) \times \GL_{n_2}(\CC)$ on $\Mat_{m_1 \times n_2}(\CC)$ (as outlined in Remark~\ref{rem:MatricesAsTensors}) gives a right action $*$ on $\Sym^{r}(\Mat_{m_1 \times {n_2}}(\CC)^{\vee})$ via 
\[
    [\beta*(A, B)](X_1 \dots X_r)=\beta(AX_1\tp{B},AX_2\tp{B}, \dots, AX_r\tp{B}),
\]
where $(A, B) \in \GL_{m_1}(\CC)\times\GL_{n_2}(\CC)$, which we make into left action by the rule 
\begin{equation} \label{eq:def rho prime}
    \rho'(A, B)(\beta)
        =
    \beta*(\tp{A}, \tp{B}).
\end{equation}

On the collection of coefficients $(\left.\partial_{x_\alpha}f\right\vert_{\substack{x=0\\y=0}})_{\alpha}$, the corresponding action is the expected action $\Sym^r(\rho_{\mathrm{std}}\boxtimes \rho_{\mathrm{std}})$ where $\rho_{\mathrm{std}}$ again stands for the standard representation of $\GL_{m_1}$ and $\GL_{n_2}$, respectively.

We have the decomposition
\begin{equation} \label{eq: Sym r decomposition Schur functor}
    \Sym^r(
        \rho_{\mathrm{std}} \boxtimes \rho_{\mathrm{std}}
    )
        =
    \bigoplus_{\lambda}
        \SS^{\lambda}(\rho_{\mathrm{std}})
            \boxtimes
        \SS^{\lambda}(\rho_{\mathrm{std}})
\end{equation}
(e.g. by \cite[Exercises~6.11(b), 4.51(b)]{FultonHarris} or \cite[Corollary~2.3.3]{Weyman})
where $\lambda$ runs over partitions of $r$ and $\SS^{\lambda}$ denotes the Schur functor. 

For each $\lambda$, denote by $$\mathrm{d}^r_{x, \lambda} f|_{\HH^{(1)}\times\HH^{(2)}}=\left.\mathrm{d}^r_{x, \lambda} f\right\vert_{\substack{x=0\\y=0}}$$ the map $\mathrm{d}^r_{x} f|_{\HH^{(1)}\times\HH^{(2)}}$ composed with the projection onto the $\SS^{\lambda}\boxtimes \SS^{\lambda}$-factor.

By a similar discussion, we define the forms $\mathrm{d}^r_{y, \lambda} f|_{\HH^{(1)}\times \HH^{(2)}}$ (the only difference being the dimensions of the matrix space, which is in this case $\Mat_{m_1 \times n_2}(\CC)$). The higher-derivative analogue of Proposition~\ref{prop:FirstDerivative} is the following.

\begin{theorem}\label{thm:HigherDerivatives}
As in the situation of Proposition~\ref{prop:FirstDerivative}, assume that the forms 
$
    \left.
        \mathrm{d}^s_x f
    \right\vert_{\HH^{(1)}\times \HH^{(2)}}
$ 
vanish for all $s$ with $0\leq s<r$.
Then for every partition $\lambda \vdash r$, the form 
$
    \left.
        \mathrm{d}^r_{x, \lambda} f
    \right\vert_{\HH^{(1)}\times \HH^{(2)}}
$ 
is a tensor product of automorphic forms of weights 
\[
    \rho^{+, \lambda}_{1, (k,l)}
        :=
    \Delta_{(m_1, n_1), (k,l)} 
        \otimes 
    (
        \SS^{\lambda}(\rho_{\mathrm{std}}) \boxtimes \mathbf{1}
    )
        =
    \left(
        \det^k \otimes \SS^{\lambda}(\rho_{\mathrm{std}})
    \right)
        \boxtimes 
    \det^l
\]
and 
\[
    \rho^{-, \lambda}_{2, (k,l)}
        :=
    \Delta_{(m_2, n_2), (k,l)} 
        \otimes 
    (
        \mathbf{1} \boxtimes \SS^{\lambda}(\rho_{\mathrm{std}})
    )
        = 
    \det^k 
        \boxtimes
    \left(
        \det^l \otimes \SS^{\lambda}(\rho_{\mathrm{std}})
    \right).
\]
Similarly, if the forms 
$
    \left.
        \mathrm{d}^s_y f
    \right\vert_{\HH^{(1)}\times \HH^{(2)}}
$ 
vanish for all $s<r$, then the form 
$
    \left.
        \mathrm{d}^r_{y, \lambda} f
    \right\vert_{\HH^{(1)}\times \HH^{(2)}}
$
is a tensor product of automorphic forms of weights 
\[
    \rho^{-, \lambda}_{1, (k,l)} 
        :=
    \Delta_{(m_1, n_1), (k,l)} 
        \otimes 
    (
        \mathbf{1} \boxtimes \SS^{\lambda}(\rho_{\mathrm{std}})
    )
        = 
    \det^k  
        \boxtimes
    \left(
        \det^l \otimes \SS^{\lambda}(\rho_{\mathrm{std}})
    \right)
\]
and 
\[
    \rho^{+, \lambda}_{2, (k,l)}
        :=
    \Delta_{(m_2, n_2), (k,l)} 
        \otimes 
    (
        \SS^{\lambda}(\rho_{\mathrm{std}}) \boxtimes \mathbf{1}
    )
        =
    \left(
        \det^k \otimes \SS^{\lambda}(\rho_{\mathrm{std}})
    \right)
        \boxtimes 
    \det^l.
\]

\end{theorem}

\begin{remark}
By convention, we identify both $\mathrm{d}^0_x f$ and $\mathrm{d}^0_y f$ with $f$ itself, so that vanishing of either of the forms $\left.\mathrm{d}^0_x f\right\vert_{\HH^{(1)}\times \HH^{(2)}}$ or $\left.\mathrm{d}^0_y f\right\vert_{\HH^{(1)}\times \HH^{(2)}}$ is equivalent to the assumption  $f|_{\HH^{(1)} \times \HH^{(2)}}=0$ of Proposition~\ref{prop:FirstDerivative}.
\end{remark}

\begin{remark}\label{rem:ThmForVariants}
The choices of coordinates used in this section, i.e. for cases (UB) and (UT), are especially useful in describing our differential operators explicitly. However, it will be a consequence of the discussion in Section~\ref{sec:ShimuraVar} that the construction described in this section is in fact coordinate-independent. In particular, the obvious statements for automorphic forms as described in Section~\ref{subsec:UTs} remain valid. 

Let us describe one particular case where the different variants of coordinates will be useful later on. Suppose that $m \neq n$ and that the factor $\U_1$ is quasi-split and of signature $(1, 1)$. We may then consider the diagonal embedding of symmetric spaces $\HH_1 \times \widetilde{\HH}_{m-1, n-1} \hookrightarrow \widetilde{\HH}_{m, n}$, and fix coordinates on $\widetilde{\HH}_{m, n}$ accordingly, i.e.
\begin{equation}\label{UTcoordinates}
\widetilde{\HH}_{m, n} \ni \begin{bmatrix} \tau \\ u\end{bmatrix}
=\begin{bmatrix} \tau_1 & x \\ z & \tau_2 \\ w & u_2 \end{bmatrix}, \;\;\; \tau_1 \in \HH_1, \; \begin{bmatrix} \tau_2 \\ u_2 \end{bmatrix} \in \widetilde{\HH}_{m-1, n-1}.
\end{equation}
Letting $y$ denote the column vector $\tp{\begin{bmatrix}z & w\end{bmatrix}},$ the operators $\mathrm{d}^r_{x, \lambda}(-)|_{\HH_1 \times \widetilde{\HH}_{m-1, n-1}}$ and $\mathrm{d}^r_{y, \lambda}(-)|_{\HH_1 \times \widetilde{\HH}_{m-1, n-1}}$  make sense and satisfy conclusions of Theorem~\ref{thm:HigherDerivatives} (producing Hermitian modular forms for the first factor and automorphic forms in the sense of Section~\ref{subsec:UTs} for the second factor).
\end{remark}

\section{Example: Restricting a Hermitian Analog of the Schottky Form}
\label{sec:example}
Hentschel and Krieg construct a Hermitian analog of the Schottky form as a suitable linear combination of Hermitian theta series of even unimodular Gaussian lattices \cite{HermitianSchottky}.
We will briefly review their construction, and then use it to construct a vector-valued automorphic form.

The three Hermitian positive definite matrices
\[\left[\begin{smallmatrix}2&0&0&0&1&1&0&1\\0&2&0&0&-1&1&-1&0\\0&0&2&0&0&1&1&-1\\0&0&0&2&-1&0&1&1\\1&-1&0&-1&2&0&0&0\\1&1&1&0&0&2&0&0\\0&-1&1&1&0&0&2&0\\1&0&-1&1&0&0&0&2\end{smallmatrix}\right],\,\left[\begin{smallmatrix}2&-1&0&-1&-1&-1&-i&1+i\\-1&2&1-i&0&0&0&0&-i\\0&1+i&2&0&0&0&0&1\\-1&0&0&2&1&1&i&-1\\-1&0&0&1&2&1&i&-1\\-1&0&0&1&1&2&i&-1\\i&0&0&-i&-i&-i&2&i\\1-i&i&1&-1&-1&-1&-i&4\end{smallmatrix}\right],\,\left[\begin{smallmatrix}2&0&1+i&i&0&0&0&0\\0&2&i&1-i&0&0&0&0\\1-i&-i&2&0&0&0&0&0\\-i&1+i&0&2&0&0&0&0\\0&0&0&0&2&0&1+i&i\\0&0&0&0&0&2&i&1-i\\0&0&0&0&1-i&-i&2&0\\0&0&0&0&-i&1+i&0&2\end{smallmatrix}\right]\]
will be denoted by $S_1$, $S_2$, and $S_3$, respectively.
Their exact values will not be important to this discussion, but they arise as the Gram matrices of the three isometry classes of even unimodular Gaussian lattices of rank 8 \cite{HermitianSchottky}.
Each Hermitian positive definite matrix $S_i$ gives rise to the Hermitian inner product $\langle v,w\rangle_i=w^\ast S_iv$ on $\ZZ[i]^8$.

We can then define Hermitian theta series
\[\Theta_i^{(n)}(\tau)=\sum_{M\in\ZZ[i]^{8\times n}}\exp(\pi i\,\tr(\tp{\overline{M}}S_i\,M\,\tau))=\sum_ha_i^{(n)}(h)\exp(2\pi i\tr(h\tau))\]
with Fourier coefficients
\begin{align}
    a_i^{(n)}(h)&=\#\{M\in\ZZ[i]^{8\times n}:\tp{\overline{M}} S_iM=2h\}\nonumber\\
    &=\#\{v_1,\ldots,v_n\in\ZZ[i]^8:\langle v_j,v_k\rangle_i=2h_{kj}\}.\label{eq:combo}
\end{align}
Hentschel and Krieg consider the linear combination $F^{(n)}=8\Theta_1^{(n)}-15\Theta_2^{(n)}+7\Theta_3^{(n)}$ and demonstrate that $F^{(4)}$ is a Hermitian analog of the Schottky form \cite{HermitianSchottky}.
\begin{lemma}\label{lem:schottky}
The linear combination 
$F^{(4)}=8\Theta_1^{(4)}-15\Theta_2^{(4)}+7\Theta_3^{(4)}$ is a nonzero cusp form of weight 8, and the restriction $F^{(4)}|_{\mathcal S_4}$ is a multiple of the Schottky form.
\end{lemma}
\begin{proof}
This is \cite[Theorem 3.1(c) and Corollary 3.4]{HermitianSchottky}.
\end{proof}
In contrast, $F^{(1)}$, $F^{(2)}$, and $F^{(3)}$ all vanish.
\begin{lemma}
    \label{lem:lowdeg}
    The linear combinations $F^{(n)}=8\Theta_1^{(n)}-15\Theta_2^{(n)}+7\Theta_3^{(n)}$ vanish for $n\leq3$.
    For $n=1$, we have $\Theta_1^{(1)}=\Theta_2^{(1)}=\Theta_3^{(1)}$.
\end{lemma}
\begin{proof}
    From the combinatorial description of $a_i^{(n)}(h)$ given in Equation \eqref{eq:combo}, we have
    \[a_i^{(n)}(h)=a_i^{(n+1)}\left(\begin{bmatrix}h\\&0\end{bmatrix}\right).\]
    In other words, each Fourier coefficient of $\Theta_i^{(n)}$ appears as a singular Fourier coefficient of $\Theta_i^{(n+1)}$.
    The same is true for the linear combination $F^{(n)}$.
    In particular, if all singular Fourier coefficients of $F^{(n+1)}$ vanish (i.e., if $F^{(n+1)}$ is a cusp form), then $F^{(n)}$ must vanish.
    Since $F^{(4)}$ is a cusp form, this shows that $F^{(n)}$ vanishes for $n\leq3$.

    Lemma 3.3(c) in \cite{HermitianSchottky} states that $\Theta_1^{(n)}|_{\mathcal S_n}=\Theta_3^{(n)}|_{\mathcal S_n}$, where $\mathcal S_n$ denotes the Siegel upper half-space of degree $n$.
    In particular, we have $\Theta_1^{(1)}=\Theta_3^{(1)}$ since $\mathcal S_1=\mathcal H_1$.
    But then the relation $8\Theta_1^{(1)}-15\Theta_2^{(1)}+7\Theta_3^{(1)}=0$ forces $\Theta_1^{(1)}=\Theta_2^{(1)}=\Theta_3^{(1)}$.
\end{proof}

\begin{remark}
\label{rem:example}
Let us be explicit about the setup for our example. We fix the field $K=\mathbb{Q}(\sqrt{-1})=\mathbb{Q}(i)$ and consider the groups $U(\eta_n)$ treated as algebraic groups with the obvious integral model $\UU(\eta_n)$, i.e. 
$$\UU(\eta_n)(A)=\{g \in \GL_{2n}(A \otimes_{\ZZ}\ZZ[i])\;|\;\tp{\overline{g}}\eta_n g=\eta_n\}, \;\; A \in \Alg_{\ZZ}.$$ 
The form $F^{(4)}$ is then Hermitian modular of weight $(0, 8)$ and full level $\Gamma=\UU(\eta_4)(\ZZ)$. 

To apply our construction, we consider the diagonal embedding of $\UU(\eta_3)\times \UU(\eta_1)$ into $\UU(\eta_4)$. It is worth noting that the standard representation $\rho_{\mathrm{std}}$ associated with the second factor as in Section~\ref{subsec:Construction} is one-dimensional. Consequently, the decomposition~(Equation \eqref{eq: Sym r decomposition Schur functor}) of $\Sym^r(\rho_{\mathrm{std}}\boxtimes\rho_{\mathrm{std}})$ in terms of Schur functors is trivial, i.e. the $\SS^{\lambda}\boxtimes\SS^{\lambda}$ terms will vanish unless $\lambda=(r)$. That is, we have $\Sym^r(\rho_{\mathrm{std}}\boxtimes\rho_{\mathrm{std}})\simeq \Sym^r(\rho_{\mathrm{std}})\boxtimes\Sym^r(\rho_{\mathrm{std}}),$ and there is no need to take any projection to a $\SS^{\lambda}\boxtimes \SS^{\lambda}$-component in our construction.
\end{remark}

We will show that $F^{(4)}$ vanishes to order 4 along $\mathcal H_3\times\mathcal H_1$ and that the fourth derivative $\mathrm{d}_x^4F^{(4)}|_{\mathcal H_3\times\mathcal H_1}$ is a nonzero vector-valued automorphic cusp form.
We will do this by explicitly computing the Fourier expansion of the derivatives $\mathrm{d}_x^rF^{(4)}|_{\mathcal H_3\times\mathcal H_1}$ and making use of the combinatorial description of $a_i^{(n)}(h)$ given in Equation \eqref{eq:combo}.

\begin{remark}
    Using the notation of Section \ref{subsec:Autoforms}, the weight of $F^{(4)}$ is the 1-dimensional representation $\Delta_{0,8} = \mathbf{1} \boxtimes \det^8$ on $\GL_4(\CC) \times \GL_4(\CC)$. Upon restriction to $\UU(\eta_3) \times \UU(\eta_1)$, its weight is $(\mathbf{1} \boxtimes \det^8) \boxtimes (\det^8 \boxtimes \mathbf{1})$ on $(\GL_3(\CC) \times \GL_3(\CC)) \times (\GL_1(\CC) \times \GL_1(\CC))$. 
    
    The representation on $\GL_1(\CC) \times \GL_1(\CC)$ corresponds to the weight of modular forms on $\UU(\eta_1) \cong \SL_2(\CC)$; in this case, modular forms of weight 8 on $\SL_2(\CC)$.

    Lastly, according to Theorem \ref{thm:HigherDerivatives} (omitting any choice of partitions $\lambda$ of $r=4$), the weight of $\mathrm{d}_x^4 F^{(4)}|_{\mathcal{H}_3 \times \mathcal{H}_1}$ and $d_y^4 F^{(4)}|_{\mathcal{H}_3 \times \mathcal{H}_1}$ are the representations
    \begin{equation} \label{eq:weight of dx4 F}
        (
            \Sym^4(\rho_{\std}) \boxtimes \det^8
        ) 
            \boxtimes 
        (
            \det^8 \boxtimes \det^4
        )
    \end{equation}
    and
    \begin{equation} \label{eq:weight of dy4 F}
        (
            \mathbf{1} 
                \boxtimes 
            (\det^8 \otimes \Sym^4(\rho_{\std})) 
        ) 
            \boxtimes 
        (
            \det^{12} \boxtimes \mathbf{1}
        )
    \end{equation}
    of $\GL_3(\CC) \times \GL_3(\CC) \times \GL_1(\CC) \times \GL_1(\CC)$ respectively. Once more, the $\GL_1(\CC) \times \GL_1(\CC)$ part corresponds to the weight of modular forms on $\UU(\eta_1) \cong \SL_2(\CC)$. In both cases, we obtain modular forms of weight 12 on $\SL_2(\CC)$.
\end{remark}

Set $c_1=8$, $c_2=-15$, and $c_3=7$.
Let $\tau_1\in\mathcal H_3$ and $\tau_2\in\mathcal H_1$.
Then the Fourier expansion of $F^{(4)}|_{\mathcal H_3\times\mathcal H_1}$ is given by
\begin{align}
    &F^{(4)}\left(\begin{bmatrix}\tau_1\\&\tau_2\end{bmatrix}\right)=\sum_ic_i\Theta_i^{(4)}\left(\begin{bmatrix}\tau_1\\&\tau_2\end{bmatrix}\right)\nonumber\\
    &=\sum_ic_i\sum_{h_1,h_2,h_3}a_i^{(4)}\left(\begin{bmatrix}h_1&\tp{\overline{h_3}}\\h_3&h_2\end{bmatrix}\right)\exp\left(2\pi i\tr\left(\begin{bmatrix}h_1&\tp{\overline{h_3}}\\h_3&h_2\end{bmatrix}\begin{bmatrix}\tau_1\\&\tau_2\end{bmatrix}\right)\right)\nonumber\\
    &=\sum_ic_i\sum_{h_1,h_2}\exp\left(2\pi i\tr(h_1\tau_1)\right)\exp(2\pi i\tr(h_2\tau_2))\sum_{h_3}a_i^{(4)}\left(\begin{bmatrix}h_1&\tp{\overline{h_3}}\\h_3&h_2\end{bmatrix}\right).\label{eq:restr}
\end{align}
More generally, we can compute
\begin{align*}
    &F^{(4)}\left(\begin{bmatrix}\tau_1&x\\y&\tau_2\end{bmatrix}\right)=\sum_ic_i\Theta_i^{(4)}\left(\begin{bmatrix}\tau_1&x\\y&\tau_2\end{bmatrix}\right)\\
    &=\sum_ic_i\sum_{h_1,h_2,h_3}a_i^{(4)}\left(\begin{bmatrix}h_1&\tp{\overline{h_3}}\\h_3&h_2\end{bmatrix}\right)\exp\left(2\pi i\tr\left(\begin{bmatrix}h_1&\tp{\overline{h_3}}\\h_3&h_2\end{bmatrix}\begin{bmatrix}\tau_1&x\\y&\tau_2\end{bmatrix}\right)\right)\\
    &=\sum_ic_i\sum_{h_1,h_2,h_3}a_i^{(4)}\left(\begin{bmatrix}h_1&\tp{\overline{h_3}}\\h_3&h_2\end{bmatrix}\right)\exp(2\pi i\tr(h_1\tau_1+h_2\tau_2+h_3x+\tp{\overline{h_3}}y)).
\end{align*}
Then the Fourier expansion of $\mathrm{d}_x^rF^{(4)}|_{\mathcal H_3\times\mathcal H_1}$ is given by
\begin{align}
    &\mathrm{d}_x^rF^{(4)}\left(\begin{bmatrix}\tau_1\\&\tau_2\end{bmatrix}\right)\nonumber\\
    &=\sum_\alpha\sum_ic_i\sum_{h_1,h_2,h_3}a_i^{(4)}\left(\begin{bmatrix}h_1&\tp{\overline{h_3}}\\h_3&h_2\end{bmatrix}\right)(2\pi i)^rh_3^\alpha\exp(2\pi i\tr(h_1\tau_1+h_2\tau_2))\,\mathrm{d}x_\alpha\nonumber\\
    &=(2\pi i)^r\sum_\alpha\sum_{h_1,h_2}\exp(2\pi i\tr(h_1\tau_1))\exp(2\pi i\tr(h_2\tau_2))\sum_ic_i\sum_{h_3}h_3^\alpha a_i^{(4)}\left(\begin{bmatrix}h_1&\tp{\overline{h_3}}\\h_3&h_2\end{bmatrix}\right)\mathrm{d}x_\alpha.\label{eq:deriv}
\end{align}
\begin{proposition}
    \label{prop:vanish0}
    The restriction $F^{(4)}|_{\mathcal H_3\times\mathcal H_1}$ vanishes.
\end{proposition}
\begin{proof}
    The combinatorial description of $a_i^{(n)}(h)$ given in Equation \eqref{eq:combo} tells us that
    \[\sum_{h_3}a_i^{(4)}\left(\begin{bmatrix}h_1&\tp{\overline{h_3}}\\h_3&h_2\end{bmatrix}\right)=a_i^{(3)}(h_1)a_i^{(1)}(h_2).\]
    Then Equation \eqref{eq:restr} and Lemma \ref{lem:lowdeg} give
    \[F^{(4)}\left(\begin{bmatrix}\tau_1\\&\tau_2\end{bmatrix}\right)=\sum_ic_i\Theta_i^{(3)}(\tau_1)\Theta_i^{(1)}(\tau_2)=F^{(3)}(\tau_1)\Theta^{(1)}(\tau_2)=0.\qedhere\]
\end{proof}
\begin{proposition}
    \label{prop:vanish4}
    The restrictions $\mathrm{d}_x^rF^{(4)}|_{\mathcal H_3\times\mathcal H_1}$ vanish for $r\not\equiv0\pmod{4}$, but the restriction $\mathrm{d}_x^4F^{(4)}|_{\mathcal H_3\times\mathcal H_1}$ does not vanish.
\end{proposition}
\begin{proof}
    Equation \eqref{eq:deriv} tells us that each Fourier coefficient of $\mathrm{d}_x^rF^{(4)}|_{\mathcal H_3\times\mathcal H_1}$ is of the form
    \[(2\pi i)^r\sum_ic_i\sum_{h_3}h_3^\alpha a_i^{(4)}\left(\begin{bmatrix}h_1&\tp{\overline{h_3}}\\h_3&h_2\end{bmatrix}\right)\]
    for fixed $h_1$, $h_2$, and $\alpha$.
    Equation \eqref{eq:combo} lets us rewrite this as
    \[(2\pi i)^r\sum_ic_i\sum_{\substack{v_1,v_2,v_3\\\langle v_j,v_k\rangle_i=2(h_1)_{kj}}}\sum_{\substack{v_4\\\langle v_4,v_4\rangle_i=2h_2}}\langle v_1,v_4\rangle_i^{\alpha_1}\langle v_2,v_4\rangle_i^{\alpha_2}\langle v_3,v_4\rangle_i^{\alpha_3}\]
    where $\alpha=(\alpha_1,\alpha_2,\alpha_3)$ with $\alpha_1+\alpha_2+\alpha_3=r$.
    
    If $r\not\equiv0\pmod{4}$, then the values of the inner sum at $v_4$, $iv_4$, $-v_4$, and $-iv_4$ will cancel with each other, so every Fourier coefficient of $\mathrm{d}_x^rF^{(4)}|_{\mathcal H_3\times\mathcal H_1}$ vanishes.

    To show that the restriction $\mathrm{d}_x^4F^{(4)}|_{\mathcal H_3\times\mathcal H_1}$ does not vanish, it is enough to find one Fourier coefficient that does not vanish.
    Set $h_1=I_3$, $h_2=1$, and $\alpha=(4,0,0)$.
    Then the Fourier coefficient in question is given by
    \[(2\pi i)^4\sum_ic_i\sum_{\substack{v_1,v_2,v_3\\\langle v_j,v_k\rangle_i=2\delta_{jk}}}\sum_{\substack{v_4\\\langle v_4,v_4\rangle_i=2}}\langle v_1,v_4\rangle^4.\]
    For each $i$, the number of vectors $v$ satisfying $\langle v,v\rangle_i=2$ is exactly 480.
    For each $i$, let $\mathcal C_i$ denote the set of these 480 vectors.
    Then we can write the sum as
    \[\sum_ic_i\sum_{v_1\in\mathcal C_i}\left[\sum_{\substack{v_2,v_3\in\mathcal C_i\\\langle v_j,v_k\rangle_i=\delta_{jk}}}1\right]\left[\sum_{v_4\in\mathcal C_i}\langle v_1,v_4\rangle^4\right]\]
    which we can compute to be exactly 1981808640.

    In order to enumerate the 480 elements of each $\mathcal C_i$, we found it helpful to use the Cholesky decomposition $S_i=d_i^{-1}\tp{\overline{L_i}}D_iL_i$, so that the problem of finding $\tp{\overline{v}}S_iv=2$ becomes the simpler problem of finding $\tp{\overline{L_iv}}D_i(L_iv)=2d_i$.
\end{proof}
\begin{theorem}\label{thm:exworks}
    The restriction $\mathrm{d}_x^4F^{(4)}|_{\mathcal H_3\times\mathcal H_1}$ is a nonzero vector-valued automorphic form.
    It can be written as a pure tensor $M\otimes\Delta$.
\end{theorem}
\begin{proof}
    Theorem \ref{thm:HigherDerivatives} and Remark \ref{rem:example} show that if the restrictions $\mathrm{d}_x^sF^{(4)}|_{\mathcal H_3\times\mathcal H_1}$ vanish for all $s<r$, then the restriction $\mathrm{d}_x^rF^{(4)}|_{\mathcal H_3\times\mathcal H_1}$ is a vector-valued automorphic form.
    Then Propositions \ref{prop:vanish0} and \ref{prop:vanish4} tell us that $\mathrm{d}_x^4F^{(4)}|_{\mathcal H_3\times\mathcal H_1}$ is a nonzero vector-valued automorphic form.
    It is a tensor product of vector-valued automorphic forms for $\mathcal U(\eta_3)$ and scalar-valued automorphic forms for $\mathcal U(\eta_1)$ of weight 12.
    Then we can write $\mathrm{d}_x^4F^{(4)}|_{\mathcal H_3\times\mathcal H_1}=M_0\otimes E_{12}+M\otimes\Delta$.
    But comparing Fourier expansions with Equation \eqref{eq:deriv} forces $M_0=0$ and $\mathrm{d}_x^4F^{(4)}|_{\mathcal H_3\times\mathcal H_1}=M\otimes\Delta$.
\end{proof}
In contrast, the restriction $F^{(4)}|_{\mathcal H_2\times\mathcal H_2}$ does not vanish.
\begin{proposition}
    The restriction $F^{(4)}|_{\mathcal H_2\times\mathcal H_2}$ does not vanish.
\end{proposition}
\begin{proof}
    Recall from Lemma \ref{lem:lowdeg} that the theta series $\Theta_i^{(2)}$ satisfy the linear relation $8\Theta_1^{(2)}-15\Theta_2^{(2)}+7\Theta_3^{(2)}=0$.
    There are no further relations since the theta series $\Theta_i^{(2)}$ span a vector space of dimension 2.
    One way to see this is to observe that the theta series $\Theta_i^{(2)}$ are not cusp forms, but  \cite[Theorem 3.1(b)]{HermitianSchottky} states that the linear combination $-8\Theta_1^{(2)}+3\Theta_2^{(2)}+5\Theta_3^{(2)}$ is a nonzero cusp form.
    
    Now suppose that the restriction $F^{(4)}|_{\mathcal H_2\times\mathcal H_2}$ did vanish.
    Then, as in the proof of Proposition \ref{prop:vanish0}, we would have have
    \[F^{(4)}\left(\begin{bmatrix}\tau_1\\&\tau_2\end{bmatrix}\right)=8\Theta_1^{(2)}(\tau_1)\Theta_1^{(2)}(\tau_2)-15\Theta_2^{(2)}(\tau_1)\Theta_2^{(2)}(\tau_2)+7\Theta_3^{(2)}(\tau_1)\Theta_3^{(2)}(\tau_2)=0.\]
    For each fixed $\tau_2$, this is a linear relation on the functions $\Theta_i^{(2)}(\tau_1)$.
    This relation must be a multiple of the relation $8\Theta_1^{(2)}-15\Theta_2^{(2)}+7\Theta_3^{(2)}=0$.
    But this would require $\Theta_1^{(2)}(\tau_2)=\Theta_2^{(2)}(\tau_2)=\Theta_3^{(2)}(\tau_2)$ for all $\tau_2$, which is false.
\end{proof}

\section{Proofs of Main Results}\label{sec:Proofs}

We now prove the main assertion, Theorem~\ref{thm:HigherDerivatives} (as well as Proposition~\ref{prop:FirstDerivative}, which is a special case). Fix all the notation ($(k, l), \U_i, \U, \HH^{(i)}, \Gamma_i, \Gamma \dots $ etc.) as in Sections~\ref{subsec:C_Restriction} and \ref{subsec:Construction}.

\subsection{Modularity} Firstly, we prove that the functions resulting from our construction obey the expected modularity rules. The key ingredient for this part of the proof is the following lemma on the differential of the action of $\Gamma$ on $\HH$.   

\begin{lemma}[{\cite[Lemma~3.4]{ShimuraArithmeticity}}]\label{lem:Shimura}
$$\mathrm{d}(\gamma \tau)={\tp{\lambda_{\gamma}(\tau)}^{-1}\mathrm{d} \tau \mu_{\gamma}(\tau)^{-1}},\;\;\; \gamma=\begin{bmatrix}A & B \\ C & D \end{bmatrix}\in \U(\RR),\;\; \tau \in \HH.$$
\end{lemma}

\begin{corollary}\label{cor:DiffirentialInCorner}
Let $\gamma=\eta(\gamma_1, \gamma_2)$ be as in Remark~\ref{rem:MatricesAsTensors}. Then for every $s>0,$ $\left.\mathrm{d}^s_x(\gamma \tau)\right\vert_{\substack{x=0\\y=0}}$ is of the form
$$\left.\mathrm{d}^s_x(\gamma \tau)\right\vert_{\substack{x=0\\y=0}}=\begin{bmatrix} 0 & * \\ 0& 0 \end{bmatrix},$$
that is, it is a matrix of symmetric forms with all forms outside of the $x$-coordinates equal to $0$. Similarly, $\left.\mathrm{d}^s_y(\gamma \tau)\right\vert_{\substack{x=0\\y=0}}$ is of the form
$$\left.\mathrm{d}^s_x(\gamma \tau)\right\vert_{\substack{x=0\\y=0}}=\begin{bmatrix} 0 & 0 \\ *& 0 \end{bmatrix},$$ where $*$ is the block of $y$-coordinates.
\end{corollary}

\begin{proof}
Let us argue for the case of $x$-coordinates only. The case $s=1$ follows directly from Lemma~\ref{lem:Shimura}, since the identity
\begin{equation}\label{eqn:ShimuraLemma}
\mathrm{d}(\gamma\tau) = {{\tp{\lambda_{\gamma}(\tau)}}^{-1}\mathrm{d} \tau \mu_{\gamma}(\tau)^{-1}}
\end{equation}
yields, after specializing to $\gamma=\eta(\gamma_1, \gamma_2),$ setting $x=y=0$ and projecting onto the $\mathrm{d}x$-coordinates, the identity
$$\left.\mathrm{d}_x (\gamma\tau)\right\vert_{\substack{x=0\\y=0}}=\begin{bmatrix}0 & \tp{\lambda_{\gamma_1}(\tau_1)}^{-1}\mathrm{d}x\,{\mu_{\gamma_2}(\tau_2)}^{-1}\\ 0 &  0  \end{bmatrix}.$$ 

The case of $s>1$ is similar, only starting with an identity obtained by differentiating Equation \eqref{eqn:ShimuraLemma} multiple times. 
\end{proof}

\begin{proposition}\label{prop:Modularity}
    In the situation of Theorem~\ref{thm:HigherDerivatives}, if 
        $\mathrm{d}_{x}^s f|_{\HH^{(1)}\times\HH^{(2)}}$ 
    vanishes for all $s<r$ then 
        $\mathrm{d}^r_{x, \lambda} f|_{\HH^{(1)}\times\HH^{(2)}}$ 
    satisfies the modularity rule
    \begin{align*}
        &\mathrm{d}^r_{x, \lambda}f|_{\HH^{(1)}\times\HH^{(2)}}(\tau_1, \tau_2)=
            \\
        &= 
        \left(
            \rho^{+, \lambda}_{1, (k,l)}(
                \lambda_{\gamma_1}(\tau_1),\mu_{\gamma_1}(\tau_1)
            )
                \otimes 
            \rho^{-, \lambda}_{2, (k,l)}(
                \lambda_{\gamma_2}(\tau_2),\mu_{\gamma_2}(\tau_2)
            )
        \right)^{-1} 
        \mathrm{d}^r_{x, \lambda} f|_{\HH^{(1)}\times\HH^{(2)}}(
            \gamma_1\tau_1, \gamma_2\tau_2
        )
    \end{align*}
    where $\tau_1 \in \HH^{(1)}, \tau_2 \in \HH^{(2)}, \gamma_1 \in \Gamma_1$ and $\gamma_2 \in \Gamma_2$. 
    
    Similarly, assuming 
        $\mathrm{d}_{y}^s f|_{\HH^{(1)}\times\HH^{(2)}}$ 
   vanishes for all $s<r$, the form 
        $\mathrm{d}^r_{y, \lambda} f|_{\HH^{(1)}\times\HH^{(2)}}$ 
    satisfies the analogous modularity rule with $\rho^{+, \lambda}_{1, (k,l)}$ replaced by $\rho^{-, \lambda}_{1, (k,l)}$ and with $\rho^{-, \lambda}_{2, (k,l)}$ replaced by $\rho^{+, \lambda}_{2, (k,l)}$.
\end{proposition}

\begin{proof}
Let us argue for the operator $\left.\mathrm{d}^r_{x} f\right\vert_{\substack{x=0\\y=0}}$ only (the proof for $\left.\mathrm{d}^r_{y} f\right\vert_{\substack{x=0\\y=0}}$ is completely analogous). Fix the element $\gamma=\eta(\gamma_1, \gamma_2) \in \Gamma$ and related notation just as in Remark~\ref{rem:MatricesAsTensors}, and note that
 for $\tau=\iota(\tau_1, \tau_2),$ we have $\gamma\tau=\iota(\gamma_1\tau_1, \gamma_2\tau_2)$ and
$$\lambda_{\gamma}(\tau)=\begin{bmatrix}\lambda_{\gamma_1}(\tau_1)& \\ & \lambda_{\gamma_2}(\tau_2)\end{bmatrix},\;\;\; \mu_{\gamma}(\tau)=\begin{bmatrix}\mu_{\gamma_1}(\tau_1)& \\ & \mu_{\gamma_2}(\tau_2)\end{bmatrix}\,.$$  Let us rewrite the modular identity
\begin{equation}\label{eqn:Modularity}
f(\tau)=\det(\lambda_{\gamma}(\tau))^{-k}\det(\mu_{\gamma}(\tau))^{-l}f(\gamma \tau)
\end{equation} 
 as 
\begin{align*}
\det(\lambda_{\gamma}(\tau))^{k}\det(\mu_{\gamma}(\tau))^{l}f(\tau)=f(\gamma \tau)
\end{align*}
Applying the operator $\left.\mathrm{d}^r_x(-)\right\vert_{\substack{x=0\\y=0}}$ then yields
\begin{align*}
\det(\lambda_{\gamma}(\tau))^{k}\det(\mu_{\gamma}(\tau))^{l}\mathrm{d}_x^r f\left(\begin{bmatrix}\tau_1  \\ & \tau_2\end{bmatrix}\right)=\left.\mathrm{d}^r_x\left(f(\gamma \tau)\right)\right\vert_{\substack{x=0\\y=0}},
\end{align*}
since all the remaining terms on the left-hand side coming from the product rule contain $\mathrm{d}_x^s f\left(\begin{bmatrix}\tau_1  \\ & \tau_2\end{bmatrix}\right)$ for some $s<r$ and hence vanish. Rearranging the resulting equation then yields
\begin{align*}
&\mathrm{d}_x^r f\left(\begin{bmatrix}\tau_1 & \\ & \tau_2\end{bmatrix}\right)=
\det(\lambda_{\gamma}(\tau))^{-k}\det(\mu_{\gamma}(\tau))^{-l}\left.\mathrm{d}^r_x\left(f(\gamma \tau)\right)\right\vert_{\substack{x=0\\y=0}}\\
&=\det(\lambda_{\gamma_1}(\tau_1))^{-k}\det(\lambda_{\gamma_2}(\tau_2))^{-k}\det(\mu_{\gamma_1}(\tau_1))^{-l}\det(\mu_{\gamma_2}(\tau_2))^{-l}\left.\mathrm{d}^r_x\left(f(\gamma \tau)\right)\right\vert_{\substack{x=0\\y=0}}\;.
\end{align*}
By the chain rule for $\mathrm{d}^r\left(f(\gamma \tau)\right)$, we have 
$$\mathrm{d}^r\left(f(\gamma \tau)\right)=\sum_{a=1}^r\sum_{b_1+b_2+\dots+b_a=r}\mathrm{d}^af \circ \left(\mathrm{d}^{b_1}(\gamma \tau),\mathrm{d}^{b_2}(\gamma \tau), \dots, \mathrm{d}^{b_a}(\gamma \tau)\right),$$ 
which gives
\begin{align*}
    \left.
        \mathrm{d}_x^r
        \left(
            f(\gamma \tau)
        \right)
    \right\vert_{\substack{x=0\\y=0}}
        &=
    \sum_{a=1}^r \sum_{b_1+b_2+\dots+b_a=r}
        \left( 
            \left. 
                \mathrm{d}^a f
            \right\vert_{\substack{x=0\\y=0}} 
        \right) 
            \circ 
        \left.
            \left(
                \mathrm{d}_x^{b_1}(\gamma \tau),
                \mathrm{d}_x^{b_2}(\gamma \tau), 
                \dots, 
                \mathrm{d}_x^{b_a}(\gamma \tau)
            \right)
        \right\vert_{\substack{x=0\\y=0}} \\
        &=
    \sum_{a=1}^r \sum_{b_1+b_2+\dots+b_a=r}
        \left(
            \left. 
                \mathrm{d}_x^a f
            \right\vert_{\substack{x=0\\y=0}} 
        \right) 
            \circ 
        \left.
            \left(
                \mathrm{d}_x^{b_1}(\gamma \tau),
                \mathrm{d}_x^{b_2}(\gamma \tau), 
                \dots, 
                \mathrm{d}_x^{b_a}(\gamma \tau)
            \right)
        \right\vert_{\substack{x=0\\y=0}}\\
        &=
    \left(
        \left.
            \mathrm{d}_x^r f
        \right\vert_{\substack{x=0\\y=0}}
    \right) 
        \circ 
    \left.
        \left(
            \mathrm{d}_x(\gamma \tau),
            \mathrm{d}_x(\gamma \tau), 
            \dots, 
            \mathrm{d}_x(\gamma \tau)
        \right)
    \right\vert_{\substack{x=0\\y=0}},  
\end{align*}
where the second equality follows from Corollary~\ref{cor:DiffirentialInCorner} and the third one from the assumption that $\left. \mathrm{d}^a_{x}f\right\vert_{\substack{x=0\\y=0}}=0$ when $a<r$. Lemma~\ref{lem:Shimura} now leads to the expression 
\begin{align*}
    &\left.
        \mathrm{d}_x^r
        \left(
            f(\gamma \tau)
        \right)
    \right\vert_{\substack{x=0\\y=0}} \\
    &=\left(
        \mathrm{d}_x^rf
    \right)
    \left(
        \begin{bmatrix}
            \gamma_1\tau_1 & \\ 
            & \gamma_2\tau_2
        \end{bmatrix}
    \right)
        \circ 
    \left(
        \tp{\lambda_{\gamma_1}(\tau_1)}^{-1}
        \mathrm{d}x\, 
        \mu_{\gamma_2}(\tau_2)^{-1}, 
        \ldots, 
        \tp{\lambda_{\gamma_1}(\tau_1)}^{-1}
        \mathrm{d}x\, 
        \mu_{\gamma_2}(\tau_2)^{-1}
    \right)\,.
\end{align*}

Then, by definition of $\rho'$ as in \eqref{eq:def rho prime}, we further have
\begin{align*}
    &\left.
        \mathrm{d}_x^r
        \left(
            f(\gamma \tau)
        \right)
    \right\vert_{\substack{x=0\\y=0}} 
        =
    \rho'(
        \lambda_{\gamma_1}(\tau_1)^{-1},
        \mu_{\gamma_2}(\tau_2)^{-1}
    )
    \left(
        \mathrm{d}_x^r f
        \left(
            \begin{bmatrix}
                \gamma_1\tau_1 & \\ 
                & \gamma_2\tau_2
            \end{bmatrix}
        \right)
    \right)\,,
\end{align*}
and altogether, we obtain
\begin{align*}
    \mathrm{d}^r_x f 
    \left(
        \begin{bmatrix}
            \tau_1 & \\ 
            & \tau_2
        \end{bmatrix}
    \right) 
        = &
    \det(\lambda_{\gamma_1}(\tau_1))^{-k}
    \det(\lambda_{\gamma_2}(\tau_2))^{-k}
    \det(\mu_{\gamma_1}(\tau_1))^{-l}
    \det(\mu_{\gamma_2}(\tau_2))^{-l} \\ 
        &\times
    \rho'(
        \lambda_{\gamma_1}(\tau_1)^{-1},
        \mu_{\gamma_2}(\tau_2)^{-1}
    )
    \mathrm{d}^r_x f 
    \left(
        \begin{bmatrix}
            \gamma_1\tau_1 & \\ 
            & \gamma_2\tau_2
        \end{bmatrix}
    \right)\,.
\end{align*}

Finally, projecting onto the $\SS^{\lambda}\boxtimes\SS^\lambda$-component of $\rho'$ yields the desired result.
\end{proof}

\subsection{Holomorphicity at cusps and tensor product decomposition}\label{subsec:HolomorphicityAtCusps}

Proposition~\ref{prop:Modularity} shows that $\mathrm{d}^r_{x, \lambda}f|_{\HH^{(1)}\times\HH^{(2)}}$ yields a vector-valued function that transforms the same way as the tensor product of automorphic forms in Theorem~\ref{thm:HigherDerivatives}. To conclude that $\mathrm{d}^r_{x, \lambda}f|_{\HH^{(1)}\times\HH^{(2)}}$ \textit{is} such a tensor product, we employ the following linear-algebraic lemma, going back to Witt \cite{Witt}.

\begin{lemma}\label{lem:WittLA}
Consider a map $F: X\times Y \rightarrow V \otimes_{\CC} W$ where $V, W$ are finite-dimensional $\CC$-vector spaces and $X, Y$ are arbitrary sets. Let $L_X$ ($L_Y$, resp.) be a chosen finite-dimensional subspace of maps $X \rightarrow V$ ($Y\rightarrow W$, resp.).  Fix a choice of basis $\{b_i\}_{i=1}^{n}$ of $V$ and $\{c_j\}_{j=1}^{m}$ of $W$, and assume that
\begin{enumerate}[(1)] 
\item{for all $y \in Y$ and all $j$, the projection of $F|_{X \times \{y\}}$ onto $V\otimes c_j \simeq V$ belongs to $L_X$,}  
\item{for all $x \in X$ and all $i$, the projection of $F|_{\{x\} \times Y}$ onto $b_i\otimes W \simeq W$ belongs to $L_Y$.}
\end{enumerate}
Then $F$ can be written in the form 
$$F=\sum_k G_k \otimes H_k, \;\; G_k \in L_X,\; H_k \in L_Y.$$ 
\end{lemma}

Before proceeding with the proof, we note that the assumptions of Lemma~\ref{lem:WittLA} are independent of the choices of bases. 

\begin{proof}
When $V$ and $W$ are one-dimensional, we may identify $V, W$ and $V\otimes W$ with $\CC$. Then the claim is the content of \cite[Satz A]{Witt}. In general, expressing all the involved vector functions as coordinate functions with respect to the bases $\{b_i\}_{i=1}^{n}$ of $V$, $\{c_j\}_{j=1}^m$ of $W$ and $\{b_i\otimes c_j\}_{i, j}$ of $V \otimes W$, resp.,the vector-valued functions $X \rightarrow V$ ($Y \rightarrow W$ and $X \times Y \rightarrow V\otimes W$, resp.) can be treated as scalar-valued functions $X \times \{1, \dots, n\}\rightarrow \CC$ ($Y \times \{1, \dots, m\}\rightarrow \CC$ and $X\times Y \times \{1, \dots, n\} \times \{1, \dots, m\}\rightarrow \CC$, resp.) in the obvious manner. This reduces the claim of the Lemma to the scalar-valued case.
\end{proof}

\begin{proposition}\label{prop:Cusps}
In the situation of Proposition~\ref{prop:Modularity}, the form $\mathrm{d}^r_{x, \lambda}f|_{\HH^{(1)}\times\HH^{(1)}}$ satisfies the assumptions of Lemma~\ref{lem:WittLA} with
$$X=\HH^{(1)},\;\; L_X=M_{\rho^{+, \lambda}_{1, (k,l)}}(\Gamma_1)\;\; \text{and }\;\;Y=\HH^{(2)},\;\; L_Y=M_{\rho^{-, \lambda}_{2, (k,l)}}(\Gamma_2)\,.$$
Similarly, the form $\mathrm{d}^r_{y, \lambda}f|_{\HH^{(1)}\times\HH^{(1)}}$ satisfies the assumptions of Lemma~\ref{lem:WittLA} with
$$X=\HH^{(1)},\;\; L_X=M_{\rho^{-, \lambda}_{1, (k,l)}}(\Gamma_1)\;\; \text{and }\;\;Y=\HH^{(2)},\;\; L_Y=M_{\rho^{+, \lambda}_{2, (k,l)}}(\Gamma_2)\,.$$
\end{proposition}

\begin{proof}
As long as neither of the unitary groups $\U_1, \U_2$ is of signature $(1, 1)$ or quasi--split over $\QQ$, to verify whether the form $\mathrm{d}^r_{x, \lambda}f|_{\HH^{(1)}\times\HH^{(1)}}$ or $\mathrm{d}^r_{y, \lambda}f|_{\HH^{(1)}\times\HH^{(1)}}$, after restriction and projection as in Lemma~\ref{lem:WittLA}, produces automorphic forms of the indicated level and weight comes down to verifying the appropriate modularity rule. In this case, the conclusion immediately follows from Proposition~\ref{prop:Modularity}.

When $\U_1$ or $\U_2$ is of signature $(1, 1)$ and is quasi-split over $\QQ$, we additionally need to verify the holomorphicity at cusps condition. Note that in this case, there is no need to take any projection to $\SS^{\lambda}$ components, and we therefore suppress $\lambda$ from the notation to simplify from now on (cf. Remark~\ref{rem:example}).
 
Let us assume that $\U_1$ is of signature $(1, 1)$ and is quasi-split over $\QQ$,  fix $\tau_2 \in \HH^{(2)}$ and let us verify the holomorphicity at cusps in the case of $\mathrm{d}^r_{x}f|_{\HH^{(1)}\times\{\tau_2\}}$ and $\mathrm{d}^r_{y}f|_{\HH^{(1)}\times\{\tau_2\}}$. The arguments are the same in the remaining cases. Acting on $f$ by $\eta(\beta)$ where $\beta \in \SL_2(\QQ),$ it is enough to verify holomorphicity at $\infty$.

We consider first the case (UT), i.e. the situation when $\U(\RR)$ is identified with $\U(\eta_n)$ and $\U$ is itself quasi-split. In this case, it is enough to even consider $\mathrm{d}^r_{x}f|_{\HH^{(1)}\times\{\tau_2\}}$ only, as the reasoning for $\mathrm{d}^r_{y}f|_{\HH^{(1)}\times\{\tau_2\}}$ is completely symmetrical. We consider the Fourier expansion of $f$ written as follows,
\begin{equation}\label{eqn:FourierExpansionInProof}
f(\tau)=\sum_h c(h)\exp({2\pi i({h_1 \tau_1}+\tr\tp{\overline{h_3}}y+\tr h_3x+\tr h_2\tau_2)}),
\end{equation}
where  $h=\begin{bmatrix}h_1 & \tp{\overline{h_3}} \\
h_3 & h_2\end{bmatrix}$ ranges over the appropriate lattice of Hermitian matrices, with $h_1$ a number and $h_2$ a block of size $(n-1, n-1)$. Then we have
$$\mathrm{d}^r_{x}f(\tau)=(2 \pi i)^r \sum_{h}\sum_{\alpha}c(h)h_3^{\alpha}\exp({2 \pi i (\tr{h_1 \tau_1}+\tr\tp{\overline{h_3}}y+\tr h_3x+\tr h_2\tau_2)})\mathrm{d}x_{\alpha},$$ where $\alpha=(i_1, i_2, \dots, i_r)$ is a multi-index and $h_3^{\alpha}$ denotes $h_3^{(i_1)}h_3^{(i_2)}\dots h_3^{(i_r)}$, the product of respective entries of the row vector $h_3$.

Consequently, we have 
\begin{align}
\left.\mathrm{d}^r_{x} f\right\vert_{\substack{x=0\\y=0}}&=(2 \pi i)^r \sum_{h}\sum_{\alpha}c(h)h_3^{\alpha}\mathrm{d}x_{\alpha}\exp({2 \pi i ({h_1 \tau_1}+\tr h_2\tau_2)})\\
&=\label{FourierForRestriction}\sum_{\alpha}\sum_{h_1} \underbrace{\left((2 \pi i)^r\sum_{h_2, h_3}c\left(\begin{bmatrix}h_1 & \tp{\overline{h_3}} \\ h_3 & h_2\end{bmatrix}\right)h_3^{\alpha}\exp({2 \pi i {\tr(h_2 \tau_2)}})\right)}_{C(h_1, \alpha)}\exp({2 \pi i {h_1 \tau_1}}) \mathrm{d} x_{\alpha},
\end{align}
where for fixed $\tau_2$ and $\alpha$, the terms $C(h_1, \alpha)$ are the Fourier coefficients for $\mathrm{d}^r_{x} f|_{\HH^{(i)}\times \{\tau_2\}}$ projected onto $\mathrm{d}x_{\alpha}$. It follows that such a coefficient indexed by $h_1$ can be nonzero only if $h_1$ fits into a positive-semidefinite Hermitian matrix $\begin{bmatrix}h_1 & \tp{\overline{h_3}} \\
h_3 & h_2\end{bmatrix}$. In particular, in this case $h_1 \geq 0$, which proves the claim.

In the case (UB), we proceed similarly using Fourier--Jacobi expansions. Let us assume $m> n$, and utilize a change of coordinates on $\U$ according to Section~\ref{subsec:UTs}. That is, we may treat $f$ as a function $f(\tau, u)$ on the symmetric space $\widetilde{\HH}_{m, n}$ instead, and consider the variant of the construction outlined in Remark~\ref{rem:ThmForVariants}. In the notation introduced therein, the Fourier--Jacobi expansion takes the following form:
\begin{equation}\label{eqn:FourierJacobiExpansionInProof}
f(\tau, u)=\sum_h c(w, u_2; h)\exp({2\pi i({h_1 \tau_1}+\tr\tp{\overline{h_3}}z+\tr h_3x+\tr h_2\tau_2)})
\end{equation}
(recall from Remark~\ref{rem:ThmForVariants} that $z, w$ are names for $y$-coordinates based on whether they come from $\tau$ or $u$). In the case of the operator $\mathrm{d}^r_x(-),$ the argument above applies almost verbatim, replacing $c(h)$ with $c(w, u_2; h)$, $\tr{\tp{\overline{h_3}y}}$ with $\tr{\tp{\overline{h_3}z}}$, etc. 

In the case of the operator $\mathrm{d}^r_y(-),$ the same argument still applies, but the formula for the resulting coefficients $C(h_1, \bullet )$ is more involved; namely, we have
$$C(h_1, \alpha, \beta)=\sum_{h_2, h_3}\partial_{w_\beta}c\left( 0, u_2; \begin{bmatrix} h_1 & h_3\\ \tp{\overline{h_3}} & h_2\end{bmatrix}\right)(2 \pi i)^{|\alpha|}h_3^{\alpha}\exp({2 \pi i {\tr(h_2 \tau_2)}})\mathrm{d} z_{\alpha}\mathrm{d} w_{\beta}.$$  
Here $\alpha, \beta$ are again multi-indices with $|\alpha|+|\beta|=r$, where $|\alpha|, |\beta|$ denotes their lengths. The key point is that when the matrix $h$ is not positive-semidefinite, the coefficient functions $c(u; h)$ are identically zero functions of $u$, and therefore so are all the partial derivatives $\partial_{w_{\beta}}c(-; h)$ appearing in the formula. 

Finally, the remaining case is when $\U$ is of equal signature $(n, n)$, but not itself quasi-split. The argument in this case uses the second variant of coordinates listed in Section~\ref{subsec:UTs}, but otherwise goes along the same lines as the above two variants. To avoid excessive repetition, we leave this case to the reader. 
\end{proof}

\begin{proof}[Proof of Theorem~\ref{thm:HigherDerivatives}]
Theorem~\ref{thm:HigherDerivatives} follows directly as a combination of Lemma~\ref{lem:WittLA} and Proposition~\ref{prop:Cusps}. Let us only stress the point that the spaces $L_X, L_Y$ taken in Proposition~\ref{prop:Cusps} are finite-dimensional, so that Lemma~\ref{lem:WittLA} applies. 
\end{proof}

We finish this section with the observation that our construction produces cusp forms out of cusp forms. 

\begin{proposition}\label{prop:CuspForms}
In the situation of Theorem~\ref{thm:HigherDerivatives}, assume that we are in the case (UT) and that $f$ is a cusp form. Then the decomposition of $\left.\mathrm{d}^r_{x, \lambda} f\right\vert_{\HH^{(1)}\times \HH^{(2)}}$ can be written in the form $\sum f_k \otimes F_k,$ where all the forms $f_k, F_k$ are cusp forms of appropriate levels and weights. Similarly, in the decomposition  $\left.\mathrm{d}^r_{y, \lambda} f\right\vert_{\HH^{(1)}\times \HH^{(2)}}=\sum g_k \otimes G_k,$  all the forms $g_k, G_k$ can be taken as cusp forms. 
\end{proposition}

\begin{proof}
We may repeat the proofs of Proposition~\ref{prop:Cusps} and Theorem~\ref{thm:HigherDerivatives} almost verbatim, with the following two adjustments:
\begin{enumerate}[(1)]
\item{In the Fourier expansion for $f$ (Equation \eqref{eqn:FourierExpansionInProof}), one has $c(h)\neq 0$ only when the Hermitian matrix $h$ is positive-definite (rather than non-negative). As a result, writing again $$h=\begin{bmatrix} h_1 & \tp{\overline{h_3}}\\ h_3 & h_1 \end{bmatrix}$$ for $h_1, h_2$ Hermitian matrices of the appropriate dimensions, the coefficients in the analogue of Equation \eqref{FourierForRestriction} are nonzero only when $h_1, h_2$ are positive-definite.}
\item{As a result, we conclude an analogue of Proposition~\ref{prop:Cusps} (hence an analogue of proof of Theorem~\ref{thm:HigherDerivatives}) with the choice of $L_X$ and $L_Y$ as the spaces of cusp forms (of the indicated level and weight) instead of the full spaces of automorphic forms.}
\end{enumerate}
\end{proof}

\section{Algebraic geometric differential operators}\label{sec:ShimuraVar}
We now explain how to reformulate our differential operators algebraic geometrically.  While unnecessary for the explicit application above, this gives a coordinate-free description of the operators that could be seen as more intrinsic.  It also shows that the construction can be realized over a smaller ring than $\CC$.  The idea for this formulation is based on the geometric construction of the Maass--Shimura operators in \cite{kaCM} that was extended to symplectic groups in \cite{hasv} and unitary groups in \cite{Eischen2012, EM21, EM22}.  

This section is divided into two portions.  First, in Section \ref{sec:algingredients}, we introduce the main ingredients in a general setting, without specialization to automorphic forms or unitary groups.  Then, in Section \ref{sec:algdiffopsmforms}, we specialize to the setting of certain Shimura varieties of type A (unitary groups), noting that type C (symplectic groups) can be obtained similarly.  The main result of this section is Theorem \ref{thm:reformulation}, which reformulates the differential operators from earlier in the paper algebraic geometrically.

\subsection{Ingredients}\label{sec:algingredients}
In this section, we introduce ingredients for our geometric reformulation of the differential operators.  These ingredients are schemes and sheaves with particular properties (\ref{sec:sheavesschemes}), the Gauss--Manin connection and Kodaira--Spencer morphism (\S\ref{sec:gmks}), algebraic differential operators on algebraic de Rham cohomology (\S\ref{sec:operatorD}), and Maass--Shimura operators (\S\ref{sec:msops}).  We will specialize to the setting of relevant Shimura varieties in Section \ref{sec:algdiffopsmforms}.  To aid with clarity, Section \ref{sec:takeaways} summarizes the key takeaways from the present section that will enable us to efficiently construct the operators in our specific setting.  For readers seeking a more detailed treatment of the ingredients introduced in this section, we recommend \cite[Section 3]{EFMV18}; other possibilities include \cite[Section 5.1]{EM21}, \cite[Section 3]{Eischen2012}, and \cite[Section 4]{hasv}.

\subsubsection{Schemes and sheaves with particular properties}\label{sec:sheavesschemes}
Given a smooth morphism of schemes $Y\rightarrow T$ and a polarized abelian scheme $\tilde{\pi}:\auniv\rightarrow Y$, we consider the Hodge bundle
\begin{align}\label{equ:hodgebundle}
\omega_{\auniv/Y}:=\tilde{\pi}_*\Omega^1_{\auniv/Y}
\end{align}
and the algebraic de Rham cohomology $H^1_{\dR}(\auniv/Y)$.  
When the data $\auniv/Y$ is clear from context, we set $\omega=\omega_{\auniv/Y}$ and $H^1_{\dR}:=H^1_{\dR}(\auniv/Y)$.  Likewise, when the data $Y/T$ is clear from context, we set $\Omega := \Omega_{Y/T} = \Omega^1_{Y/T}$.  We identify $\omega$ with its image in $H^1_{\dR}$ under the Hodge filtration
\begin{align*}
0\rightarrow \omega\hookrightarrow H^1_{\dR}\twoheadrightarrow R^1\tilde{\pi}_*\mathcal{O}_{\auniv}\rightarrow 0.
\end{align*}

Given a $T$-subscheme $\iota: Y'\hookrightarrow Y$, we have exact sequences of sheaves
\begin{align}
0\rightarrow \mathcal{N}^\vee_{Y'/Y}\rightarrow \iota^\ast\Omega_{Y/T}\rightarrow \Omega_{Y'/T}\rightarrow 0\label{conormalsplit}\\
0\rightarrow \mathcal{T}_{Y'/T}\rightarrow \iota^\ast \mathcal{T}_{Y/T}\rightarrow\mathcal{N}_{Y'/Y}\rightarrow 0,\nonumber
\end{align}
where $\mathcal{T}_{Y/T}\cong \Omega_{Y/T}^\vee$ and $\mathcal{T}_{Y'/T}= \Omega_{Y'/T}^\vee$ are tangent sheaves and $\mathcal{N}_{Y'/Y}$ denotes the normal sheaf on $Y'$.  The sheaf $\mathcal{N}^\vee_{Y'/Y}$ is the conormal bundle.  In general, these exact sequences do not split.  We will see in Section \ref{sec:conormalsplitting} that there is a canonical splitting, however, for the particular embeddings of Shimura varieties of types A and C arising in Section \ref{sec:operatorD}.

\subsubsection{Gauss--Manin connection and Kodaira--Spencer morphism}\label{sec:gmks}
We will construct differential operators from the Gauss--Manin connection
\begin{align*}
\nabla = \nabla_{\auniv/Y}: H^1_{\dR}\rightarrow H^1_{\dR}\otimes\Omega_{Y/T}^1
\end{align*}
and the Kodaira--Spencer morphism $\omega\otimes\omega\rightarrow\Omega_{Y/T}^1$.  Since $\nabla$ is a connection, we have 
\begin{align}\label{equ:nablaparts}
\nabla(fu) = u\otimes df + f\nabla(u)
\end{align}
for all $f$ in the structure sheaf of $Y$ and sections $u$ in $H^1_{\dR}$.  As noted in Section \ref{sec:ksiso}, in the setting of Shimura varieties of types A and C, the Kodaira--Spencer morphism induces an isomorphism 
\begin{align}\label{equ:ksiso1}
\ks:\Omega^1\xrightarrow{\sim}\omega^2, 
\end{align}
where $\omega^2$ is a certain subsheaf of $\omega^{\otimes 2}$.  When we have such an isomorphism $\ks$, we identify $\Omega^1$ and $\omega^2$ through $\ks$.

Via the Leibniz rule (product rule), $\nabla$ extends to a connection on tensor powers, symmetric powers, and exterior powers of $\omega$ and of $H^1_{\dR}$.  More generally, this also extends further to include sheaves obtained by applying Schur functors, like in \cite{EFMV18}.  We will denote such sheaves by $\mathcal{F}$ immediately below.

\subsubsection{Algebraic differential operators on algebraic de Rham cohomology}\label{sec:operatorD}
When we have an isomorphism $\ks$ as in \eqref{equ:ksiso1}, we compose $\nabla$ with $1\otimes\ks$ to obtain algebraic differential operators 
\begin{align*}
D: H^1_{\dR}&\rightarrow H^1_{\dR}\otimes \omega^2.
\end{align*}
We can iterate the operator $D$, $d$ times for each positive integer $d$, to obtain differential operators $D^d = D\circ\cdots \circ D$ that are applied to $H^1_{\dR}$ and, more generally, to sheaves $\mathcal{F}$ formed from $H^1_{\dR}$ as described immediately above. 

Given a $T$-subscheme $\iota: Y'\hookrightarrow Y$ as above, we denote by $\mathcal{F}_r$ the subsheaf of sections of $\mathcal{F}$ that vanish to order $r$ on $Y'$, i.e. sections $\sum_{i} f_i u_i$ with the $u_i$ a basis of nowhere vanishing sections in $\mathcal{F}$ and the $f_i$ elements of the structure sheaf such that $f_i$ vanishes to order $r$ on $Y'$ for all $i$. 
\begin{lemma}\label{lem:Dbasic1}
Suppose we have the isomorphism $\ks$ on $Y$ as above.  For $\mathcal{F}$ and $\iota$ as immediately above, we have $\iota^\ast(D^r\mathcal{F}_r)\subset \iota^\ast(\mathcal{F}\otimes \Sym^r\Omega_{Y/T})$.
\end{lemma}
\begin{proof}
This follows immediately from applying Equation \eqref{equ:nablaparts} to a section of $\mathcal{F}$ that vanishes to order $r$ on $Y'$.
\end{proof}

Suppose we have the isomorphism $\ks$ on $Y$ and an embedding $\iota:Y'\hookrightarrow Y$ as above.  If, additionally, the exact sequence \eqref{conormalsplit} splits, then we denote
\begin{align*}
\pi:\iota^\ast\Omega_{Y/T}\rightarrow \mathcal{N}^\vee
\end{align*} 
the projection onto $\mathcal{N}^\vee$ mod $\Omega_{Y'/T}$.  We also use the same notation for the induced projections on symmetric powers
\begin{align*}
\pi:\iota^\ast\Sym^d\Omega_{Y/T}&\rightarrow \Sym^d(\mathcal{N}^\vee).
\end{align*}
Furthermore, we write $\pi$ to mean the projection $\mathrm{id}_{\iota^\ast\mathcal{F}}\otimes\pi$, where $\mathrm{id}_{\iota^\ast\mathcal{F}}$ denotes the identity on $\iota^\ast\mathcal{F}$.  This simplifies notation, and there will be no ambiguity about the meaning in the contexts in which we will employ this notation.  Under these conditions, we define an operator
\begin{align}\label{equ:thetadefn}
\Theta^r:=\pi\circ\iota^\ast\circ \left(D^r|_{\mathcal{F}_r}\right): \mathcal{F}_r\rightarrow \iota^\ast\left(\mathcal{F}\otimes\Sym^r\mathcal{N}^\vee\right).  
\end{align}
Explicitly, if $fu$ is a section of $\mathcal{F}_r$ with $u$ a nonvanishing differential and $f$ in the structure sheaf, $\mathscr{B} = \left\{\partial_i\right\}_{1\leq i \leq m}$ a basis for the tangent bundle, and $\mathscr{B}'=\left\{w_i\right\}_{1\leq i \leq m}$ the dual basis for the cotangent bundle, then 
\begin{align}\label{equ:theyrethesame}
\iota^* \circ (D^r|_{\mathcal{F}_r}) (fu) = \iota^\ast\left(u\otimes \sum_{1\leq\nu_1\leq \cdots\leq \nu_r\leq m}\partial_{\nu_1}\cdots\partial_{\nu_r}(f)w_{\nu_1}\cdots w_{\nu_r}\right).
\end{align} 

\subsubsection{Brief digression on Maass--Shimura operators}\label{sec:msops}
Although Maass--Shimura operators are not the main focus of this paper, it will be useful for us to recall their construction briefly.  (Much more detailed treatments are available in \cite{Eischen2012, hasv, kaCM}.)  Suppose we, for the moment, extend our consideration from the algebraic to the $C^\infty$ setting and take $Y$ to be a manifold over which the Hodge decomposition $H^1_{\dR} = H^{1, 0}\oplus H^{0, 1}$ holds, with $\omega = H^{1, 0}$.  Still working in the $C^\infty$ setting, the projection of $H_{\dR}^1$ onto $\omega$ mod $H^{0,1}$ induces projections from $\left(H^1_{\dR}\right)^{\otimes d}$ to $\omega^{\otimes d}$, and we have analogous results for the sheaves $\mathcal{F}$.  Suppose we have the isomorphism $\ks$ as above.  Let $\mathscr{D}$ be the operator obtained by composing the algebraic operator $D$ from Section \ref{sec:operatorD} with this projection.  We have that $H^{0, 1}$ is holomorphically horizontal, i.e. $\nabla(H^{0, 1})\subseteq H^{0, 1}$.  Thus, $D$ commutes with taking quotients mod $H^{0, 1}$, and it makes sense to iterate the operator $\mathscr{D}$ $d$ times to obtain operators $\mathscr{D}^d$.  In the setting of where $\mathcal{F}$ is a sheaf of automorphic forms on a Shimura variety $Y$ over $\CC$, this operator is called the {\em Maass--Shimura operator}. 
In general, sections in the image of $\mathscr{D}^d$ are merely $C^\infty$, not holomorphic.  As a corollary of Lemma \ref{lem:Dbasic1} concerning the operator $D$, we have the following result for an embedding $\iota: Y'\hookrightarrow Y$ of manifolds analogous to the one for sheaves above.
\begin{corollary}\label{coro:comparems}
Suppose we have the isomorphism $\ks$ on $Y$ as above.  For $\mathcal{F}$ and $\iota$ as immediately above, we have $\iota^\ast(\mathscr{D}^r\mathcal{F}_r)=\iota^\ast(D^r\mathcal{F}_r)\subset \iota^\ast(\mathcal{F}\otimes \Sym^r\Omega_{Y/T})$.
\end{corollary}
\begin{proof}
This follows immediately from the proof of Lemma \ref{lem:Dbasic1}, together with the realization of $\mathscr{D}$ as the quotient of $D$ mod $H^{0, 1}$.
\end{proof}

\subsubsection{Takeaways}\label{sec:takeaways}
In Section \ref{sec:algdiffopsmforms}, we employ the above ingredients in a special setting: automorphic forms on unitary Shimura varieties.  The operator $\Theta^r$ and a variant we produce below will be our desired differential operators on automorphic forms.  Thus, we need to select specific instances of the following:
\begin{itemize}
\item{Schemes $\mathcal{A}$, $Y$, and $T$ and an embedding of $T$-schemes $\iota:Y'\hookrightarrow Y$ meeting the above criteria}
\item{A sheaf $\mathcal{F}$ on $Y$ meeting the above criteria}
\end{itemize}
such that the following hold:
\begin{itemize}
\item{The sections of $\mathcal{F}$ can be identified with automorphic forms on unitary groups, and the sections of $\iota^\ast\mathcal{F}$ correspond to automorphic forms on the desired product of unitary groups.}
\item{Over $\CC$, $\iota: Y'\hookrightarrow Y$ can be identified with the embeddings of Hermitian symmetric spaces from the first part of this paper, and $\mathcal{F}$ is identified with the space of automorphic forms. (N.B. Over these Hermitian symmetric spaces, we have the Hodge splitting of $H^1_{dR}$, so we do not need to check this criterion separately.)}
\item{The exact sequence \eqref{conormalsplit} splits.}
\end{itemize}
Once we have all this, as well as an additional splitting of $\mathcal{N}^\vee$, Equation \eqref{equ:theyrethesame} will enable us to see that we have produced algebraic differential operators on algebraic geometric automorphic forms that agree, over $\CC$, with the differential operators produced earlier in this paper.

\subsection{Differential operators on algebraic automorphic forms}\label{sec:algdiffopsmforms}
We now specialize the constructions and results from Section \ref{sec:algingredients} to the setting of automorphic forms on certain Shimura varieties of type A (unitary groups) and C (symplectic groups).  The first three portions of this section recall material that is already well-covered in the existing literature (or straightforward to deduce from the existing literature): Section \ref{sec:shimuravarieties} briefly recalls key properties of certain Shimura varieties of types A and C (including those that will play the role of $Y'$ and $Y$ from above in our setting), Section \ref{sec:shimurasheaves} introduces some well-known sheaves, and Section \ref{sec:ksiso} states the Kodaira--Spencer isomorphism $\ks$ for such Shimura varieties.  The only new material in this section is covered in the last two portions: Section \ref{sec:conormalsplitting}, which establishes splittings for the conormal bundle $\mathcal{N}^\vee$, and Section \ref{sec:diffopcomparison}, which completes the construction of our algebraic geometric differential operators and proves that they coincide over $\CC$ with the differential operators defined earlier in the paper.

We keep the background on Shimura varieties concise, recalling only the details necessary to move ahead.  Without this approach, it would be easy for the reader to get lost in well-established information about Shimura varieties and miss the main points about what is actually new here, namely the differential operators.  In case the reader would like a more thorough treatment of the background material, though, we cite references that go into much more detail and generality.

\subsubsection{Some Shimura varieties}\label{sec:shimuravarieties}
We introduce certain Shimura varieties of type A (unitary) and C (symplectic), whose components over $\CC$ can be identified with the Hermitian symmetric spaces from earlier in the paper.  Everything in this section is covered in more detail and generality in the literature, e.g. \cite{Lan20, kot92, Lan13}.  The literature closest to the presentation here includes \cite[Sections 2.2 and 4.4.1]{EischenAWS}, \cite[Sections 2.3 and 3.1]{EHLS}, and \cite[Sections 2.1 and 2.2]{EM22}.

We shall handle the cases of unitary groups (the setting of the present paper) and symplectic groups (the setting of Cl\'ery and van der Geer's work  \cite{CleryVDGeer} upon which we build) simultaneously.  Following the usual conventions, we refer to the unitary case as type A and the symplectic case as type C.  For clarity and simplicity, we only introduce here the particular instances of type A and C Shimura varieties whose connected components over $\CC$ correspond to the Hermitian symmetric spaces in our and Cl\'ery--van der Geer's work.  It is straightforward to extend our entire construction and results to all type A and C Shimura varieties, though.

For the remainder of the paper, we fix a choice of setting: either unitary groups (where our results earlier in the paper) or symplectic groups (like in \cite{CleryVDGeer}).  We refer to these as type A and C, respectively.  Furthermore, we fix a field $K$, $K$-vector spaces $V_\alpha$ and $V_\beta$, and nondegenerate pairings $\langle, \rangle_\alpha$ and $\langle, \rangle_\beta$ on $V_\alpha$ and $V_\beta$, respectively, that meet the following conditions:
\begin{itemize}
\item{For type A, we require that $K$ is a quadratic imaginary extension of $\QQ$ and that the pairings are Hermitian.}
\item{For type C, we require that $K=\QQ$ and that the pairings are symplectic.}
\end{itemize}
We set $W = V_\alpha\oplus V_\beta$ and denote by $\langle, \rangle$ the pairing on $W$ defined by $\langle(v_\alpha, v_\beta), (v_\alpha', v_\beta')\rangle = \langle v_\alpha, v_\alpha'\rangle_\alpha +\langle v_\beta, v_\beta'\rangle_\beta$ for all $v_\alpha, v_\alpha\in V_\alpha$ and $v_\beta, v'_\beta\in V_\beta$.  From this data, we obtain PEL Shimura data and associated schemes $\mathcal{M}_\alpha$, $\mathcal{M}_\beta$, and $\mathcal{M}_{U}$, which correspond to the (unitary or symplectic) groups that fix the pairings on each of these vector spaces.  Let $E$ be the compositum of the reflex fields here, so these schemes are all defined over $\mathrm{Spec}(E)$.  The field $E$ is a subfield of $K$.

Shimura varieties in this context correspond to similitude groups.  Given a $K$-vector space $V$ and a pairing $\langle, \rangle_V$ on $V$, we denote by $G(V, \langle, \rangle_V)$ the subgroup of $\GL(V)$ preserving $\langle, \rangle_V$ up to similitude.  We set $G_\alpha = G(V_\alpha, \langle, \rangle_\alpha)$, $G_\beta = G(V_\beta, \langle, \rangle_\beta)$, and $G_U = G(W, \langle, \rangle)$.  We denote by $G_{\alpha, \beta}$ the subgroup of $G_\alpha\times G_\beta$ consisting of elements $(g_\alpha, b_\beta)$ with the same similitude on each of the two factors.  There is an associated scheme $\mathcal{M}_{\alpha, \beta}$ defined over $\mathrm{Spec}(E)$.  

Associated to the canonical inclusions and projections of these groups, we have morphisms of schemes
\begin{align*}
\iota: \mathcal{M}_{\alpha, \beta}&\hookrightarrow \mathcal{M}_U\\
j: \mathcal{M}_{\alpha, \beta}&\hookrightarrow \mathcal{M}_\alpha\times_E\mathcal{M}_\beta\\
 \mathcal{M}_\alpha\times_E\mathcal{M}_\beta&\twoheadrightarrow \mathcal{M}_\alpha\\
  \mathcal{M}_\alpha\times_E\mathcal{M}_\beta&\twoheadrightarrow \mathcal{M}_\beta.
\end{align*}
We denote the composition of $j$ with the last two projections by
\begin{align*}
p_\alpha:\mathcal{M}_{\alpha, \beta}&\rightarrow\mathcal{M}_\alpha\\
p_\beta:\mathcal{M}_{\alpha, \beta}&\rightarrow\mathcal{M}_\beta.
\end{align*}
For each of the subscripts $\square$ on $\mathcal{M}$, we have universal abelian schemes 
\begin{align*}
\tilde{\pi}_\square:\mathcal{A}_\square\rightarrow \mathcal{M}_\square.
\end{align*}
When it will not cause confusion, we drop the subscript.

We make three remarks about the relationship with the material earlier in this paper:
\begin{itemize}
\item{The morphism $\iota$ is precisely the specialization to our setting of the morphism $\iota$ introduced in general in Section \ref{sec:algingredients}, i.e. we take $Y=\mathcal{M}_U$ and $Y'=\mathcal{M}_{\alpha, \beta}.$  In the notation of Section \ref{sec:algingredients}, we take $T = \mathrm{Spec} (E)$, or we extend scalars and work over a $\mathrm{Spec}(E)$-scheme $T$.}  
\item{$\mathcal{M}_U(\CC)$ is a finite union of disjoint copies of the Hermitian symmetric space $\mathfrak{h}_U$ for the subgroup $U$ of $\GL(W)$ preserving the pairing $\langle, \rangle$.  In type A, $U$ is a unitary group and $\mathfrak{h}_U =\mathcal{H}:=\mathcal{H}_U$; and in type C, $U\cong \Sp_g$ and $\mathfrak{h}_U=\mathfrak{H}_g$ for $g=\dim W$.}
\item{$\mathcal{M}_{\alpha, \beta}(\CC)$ is a finite union of disjoint copies of $\mathfrak{h}_\alpha\times\mathfrak{h}_\beta$, with $\mathfrak{h}_\alpha$ and $\mathfrak{h}_\beta$ the Hermitian symmetric spaces for the subgroups $U_\alpha$ of $\GL(V_\alpha)$ and $U_\beta$ of $\GL(V_\beta)$ preserving the pairings $\langle, \rangle_\alpha$ and $\langle, \rangle_\beta$, respectively.  In type A, this corresponds to the embedding $\mathcal{H}_\alpha\times\mathcal{H}_\beta\hookrightarrow\mathcal{H} =\mathcal{H}_U$ associated to the inclusion $U_\alpha\times U_\beta\hookrightarrow U$, like in the first part of this paper.  In type B, this corresponds to the embedding $\mathfrak{H}_j\times\mathfrak{H}_{g-j}\hookrightarrow\mathfrak{H}_g$, with $j=\dim V_\alpha$ and $g-j = \dim V_\beta$, associated to the inclusion $\Sp_j\times \Sp_{g-j}\hookrightarrow \Sp_g$ like in \cite{CleryVDGeer}.}
\end{itemize}

\subsubsection{Some sheaves}\label{sec:shimurasheaves}
For each of the subscripts $\square$ on $\mathcal{M}$, we set
\begin{align*}
\omega_\square:=\omega_{\mathcal{A}_\square/\mathcal{M}_\square},
\end{align*}
where the right hand side employs the notation introduced for the Hodge bundle in Equation \eqref{equ:hodgebundle}.
When it will not cause confusion, we drop the subscript.  We also set
\begin{align*}
H:=H_U&:=H^1_{\dR}(\mathcal{A}_U/\mathcal{M}_U).
\end{align*}

In type A, complex conjugation on $K$ induces decompositions
\begin{align}
\omega &= \omega^+\oplus \omega^-\\
H &= H^+\oplus H^-,
\end{align}
with the $\pm$-submodules determined by the two possible actions of $K$ (i.e. acting through the involution or its square).  We also have similar decompositions for any modules in the setting of type A, and we denote the superscripts similarly.

Given a vector bundle $\mathcal{G}$, we write $\wedge^{\mathrm{top}}\mathcal{G}$ for the top exterior power of $\mathcal{G}$, i.e. $\wedge^{\mathrm{top}}\mathcal{G} = \wedge^r\mathcal{G}$, where $r$ denotes the rank of $\mathcal{G}$.  Let $k$ and $\ell$ be integers.  For type $A$, a scalar-valued automorphic form of weight $(k, \ell)$ is a global section of the line bundle 
\begin{align*}
\omega^{(k, \ell)} := (\wedge^{\mathrm{top}}\omega^+)^k\otimes(\wedge^{\mathrm{top}}\omega^-)^\ell.
\end{align*}  For type $C$, a scalar-valued automorphic form of weight $k$ is a global section of the line bundle 
\begin{align*}
\omega^k:=(\wedge^{\mathrm{top}}\omega)^k.  
\end{align*}

More generally, automorphic forms for type A and C are defined in, for example, \cite[Section 3.3]{EischenAWS} and \cite[Section 2.2]{EM22}.  We briefly recall the definition now.  When we remove the requirement that we are working with scalar weight forms, the definition of automorphic forms employs the sheaves $\mathcal{E} := \mathrm{Isom}_{\mathcal{O}_\mathcal{M}}(\omega, \mathcal{O}^a_\mathcal{M})$ for type C and $\mathcal{E}: = \mathcal{E}_+\oplus\mathcal{E}_-$ with $\mathcal{E}_\pm:=\mathrm{Isom}_{\mathcal{O}_\mathcal{M}}(\omega_{\pm}, \mathcal{O}^{a_\pm}_\mathcal{M})$ for type A.  (Here, $\mathcal{O}_M$ denotes the structure sheaf of $\mathcal{M}$,  $a_\pm$ is the rank of $\omega_\pm$, and $a$ is the rank of $\omega$.)  Given a representation $(\rho, M)$ of $H = \GL_a$ or $H=\GL_{a_+}\times GL_{a_-}$ over $E$, we define
\begin{align*}
\omega^\rho :=\mathcal{E}^\rho:= \mathcal{E}\times^H M,
\end{align*}
to be the sheaf such that for each $E$-algebra $R$, $\omega^\rho(R) = (\mathcal{E}(R)\times M\otimes R)/\sim$, with the equivalence relation $\sim$ given by $(\ell, m)\sim (g\ell, \rho ({ }^t g^{-1})m)$ for all $g\in H$.  An automorphic form is a global section of $\omega^\rho = \mathcal{E}^\rho$.  Motivation for this definition is provided in \cite[Remark 3.2.8]{EischenAWS}.  Note that this construction is compatible with extending scalars.  Over $\CC$, automorphic forms defined this way can be identified with automorphic forms defined earlier in the paper.  It is straightforward to see that the Maass--Shimura differential operators defined in Section \ref{sec:msops} also respect this compatibility (which is also addressed in more detail \cite[Remark 8.1]{Eischen2012}).

\subsubsection{Kodaira--Spencer isomorphism}\label{sec:ksiso}
Over Shimura varieties of types A and C, the Kodaira--Spencer morphism induces an isomorphism
\begin{align*}
\ks: \Omega\cong\omega^2,
\end{align*}
where 
\begin{align}\label{equ:omega2}
\omega^2 :=\begin{cases}
\omega^+\otimes \omega^- & \mbox{for type A (unitary groups)}\\
\Sym^2\omega & \mbox{for type C (symplectic groups)}. 
\end{cases} 
\end{align}
Going forward, we identify $\Omega$ with $\omega^2$ via $\ks$.  More details about the Kodaira--Spencer isomorphism in our settings is available in, for example, \cite[Section 3.1]{EM22} and \cite[Section 2.3.5]{Lan13}.

\subsubsection{Splitting for conormal bundle}\label{sec:conormalsplitting}
For our embeddings of Shimura varieties of type A or C, the Kodaira--Spencer isomorphism induces a canonical splitting of the conormal bundle introduced in Section \ref{sec:algingredients}.

\begin{lemma}\label{lem:conormalsplitting}
We have a canonical splitting
\begin{align*}
\iota^\ast\Omega_{\mathcal{M}_U/T} = \Omega_{\mathcal{M}_{\alpha, \beta}/T}\oplus \mathcal{N}^\vee_{\mathcal{M}_{\alpha, \beta}/\mathcal{M}_U},
\end{align*} and furthermore,
\begin{align*}
\mathcal{N}^\vee_{\mathcal{M}_{\alpha, \beta}/\mathcal{M}_U} = \begin{cases}
\mathcal{N}^\vee_+\oplus \mathcal{N}^\vee_-, &\mbox{type A (unitary groups)}\\
p_\alpha^*\omega_\alpha\otimes p_\beta^*\omega_\beta, & \mbox{type C (symplectic groups)} ,
\end{cases}
\end{align*}
with
\begin{align*}
\mathcal{N}^\vee_+\cong (p_\alpha^\ast\omega_\alpha^+\otimes p_\beta^\ast\omega^-_\beta)\\
\mathcal{N}^\vee_-\cong (p_\alpha^\ast\omega_\beta^+\otimes p_\beta^\ast\omega^-_\alpha).
\end{align*}
\end{lemma}
\begin{proof}
 We put $\omega:= \omega_{\mathcal{A}_U/\mathcal{M}_U}$.  Via the Kodaira--Spencer isomorphism from Section \ref{sec:ksiso}, we have
\begin{align*}
\Omega_{\mathcal{M}_U/T}\cong\omega^2\\
\Omega_{\mathcal{M}_\alpha/T}\cong\omega_\alpha^2\\
\Omega_{\mathcal{M}_\beta/T}\cong\omega_\beta^2.
\end{align*}
We also have
\begin{align*}
\Omega_{\mathcal{M}_{\alpha, \beta}/T} &= p_\alpha^\ast \Omega_{\mathcal{M}_\beta/T}\oplus p_\beta^\ast\Omega_{\mathcal{M}_\beta/T}\\
\iota^\ast\omega &= p_\alpha^\ast\omega_\alpha\oplus p_\beta^\ast\omega_\beta.
\end{align*}
For type A, we additionally have
\begin{align*}
\iota^\ast\omega^\pm &= p_\alpha^\ast\omega^\pm_\alpha\oplus p_\beta^\ast\omega^\pm_\beta.
\end{align*}
Putting this together and setting 
\begin{align*}
\mathcal{L}_{\alpha, \beta} = \begin{cases}
(p_\alpha^\ast\omega_\alpha^+\otimes p_\beta^\ast\omega^-_\beta)\oplus (p_\beta^\ast\omega_\beta^+\otimes p_\alpha^\ast\omega^-_\alpha), & \mbox{type A}\\
p_\alpha^*\omega_\alpha\otimes p_\beta^*\omega_\beta, & \mbox{type C},
\end{cases}
\end{align*}
we get
\begin{align}
\iota^\ast\Omega_{\mathcal{M}_U/T} &\cong \iota^\ast(\omega^2)\nonumber\\
&\cong p_\alpha^\ast(\omega_\alpha^2)\oplus p_\beta^\ast(\omega_\beta^2)\oplus \mathcal{L}_{\alpha, \beta}\label{equ:hsplitting}\\
&\cong p_\alpha^\ast \Omega_{\mathcal{M}_\beta/T}\oplus p_\beta^\ast\Omega_{\mathcal{M}_\beta/T}\oplus \mathcal{L}_{\alpha, \beta}\nonumber \\
&\cong \Omega_{\mathcal{M}_{\alpha, \beta}/T}\oplus \mathcal{L}_{\alpha, \beta}.\nonumber
\end{align}
\end{proof}

Following the convention from Section \ref{sec:operatorD}, we denote by 
\begin{align*}
\pi:\iota^\ast\Omega_{\mathcal{M}_U/T}\rightarrow \mathcal{N}^\vee
\end{align*} 
the projection onto $\mathcal{N}^\vee$ mod $\Omega_{\mathcal{M}_{\alpha, \beta}/T}$.  In the unitary setting (type A), we denote by 
\begin{align*}
\pi_\pm: \iota^\ast\Omega_{\mathcal{M}_U/T}\rightarrow \mathcal{N}^\vee_{\pm}
\end{align*}
the projections onto $\mathcal{N}^\vee_{\pm}$ mod $\Omega_{\mathcal{M}_{\alpha, \beta}/T}\oplus\mathcal{N}^\vee_{\mp}$.  We also use the same notation for the induced projections on symmetric powers
\begin{align*}
\pi:\iota^\ast\Sym^d\Omega_{\mathcal{M}_U/T}&\rightarrow \Sym^d(\mathcal{N}^\vee)\\
\pi_\pm: \iota^\ast\Sym^d\Omega_{\mathcal{M}_U/T}&\rightarrow\Sym^d( \mathcal{N}^\vee_{\pm}).
\end{align*}
Furthermore, similarly to the convention we established for $\pi$ in Section \ref{sec:operatorD}, we write $\pi_\pm$ to mean the projection $\mathrm{id}_{\iota^\ast\mathcal{F}}\otimes\pi_\pm$, where $\mathrm{id}_{\iota^\ast\mathcal{F}}$ denotes the identity on $\iota^\ast\mathcal{F}$.  This simplifies notation, and there will be no ambiguity about the meaning in the contexts in which we will employ this notation.

\subsubsection{Algebraic differential operators on automorphic forms on unitary (and symplectic) groups}\label{sec:diffopcomparison}
For the remainder of the paper:
\begin{itemize}
\item{$k$ and $\ell$ are integers.}
\item{$\mathcal{F}$ is the sheaf $\omega^k$ or the sheaf $\omega^{(k, \ell)}$ on $\mathcal{M}_U$, depending on whether $\mathcal{M}_U$ is of type A or C.}
\item{$\iota: \mathcal{M}_{a,b}\hookrightarrow \mathcal{M}_U$ is the canonical embedding of Shimura varieties of type A (unitary groups) or C (symplectic groups) introduced above.}
\item{$\mathcal{F}_r$ denotes the subsheaf of $\mathcal{F}$ whose sections vanish to order $r$ on $\mathcal{M}_{a,b}$.}
\end{itemize}
We have the following algebraic differential operators that take automorphic forms that vanish to order $r$ on $\mathcal{M}_U$ to automorphic forms of higher weight on $\mathcal{M}_{a, b}$:
\begin{align}\label{equ:Thetadefn}
\Theta^r:\mathcal{F}_r\rightarrow \iota^\ast \mathcal{F}\otimes\Sym^r\mathcal{N}^\vee,
\end{align}
which was defined in Equation \eqref{equ:thetadefn} by $\Theta^r := \pi\circ \iota^\ast\circ\left(D^r|_{\mathcal{F}_r}\right)$.
For type A, there is also an algebraic differential operator
\begin{align}\label{equ:Thetadefnpm}
\Theta^r_\pm:\mathcal{F}_r\rightarrow \iota^\ast\mathcal{F}\otimes\Sym^r\mathcal{N}^\vee_{\pm}
\end{align}
defined by $\Theta^r_\pm := \pi_\pm\circ\iota^\ast\circ\left(D^r|_{\mathcal{F}_r}\right)$.
\begin{remark}
A reminder about notation: $D$ denotes the algebraic differential operator defined on de Rham cohomology in Section \ref{sec:operatorD}, and $\mathscr{D}$ denotes the Maass--Shimura operator from Section \ref{sec:msops}.
\end{remark}

\begin{theorem}\label{thm:reformulation}
The algebraic differential operators $\Theta^r$ and $\Theta^r_\pm$ coincide over $\CC$ with the other differential operators defined in this paper, in the following sense: 
\begin{enumerate}
\item{ $\Theta^r=\pi\circ\iota^\ast\circ\left(\mathscr{D}^r|_{\mathcal{F}_r}\right)$}
\item{$\Theta^r_\pm= \pi_\pm\circ\iota^\ast\circ\left(\mathscr{D}^r|_{\mathcal{F}_{r, \pm}}\right)$}
\item{On global sections of $\mathcal{F}_r$, $\Theta^r_\pm$ is the holomorphic operator defined in the first part of this paper, and $\Theta^r$ is the holomorphic differential operator defined \cite{CleryVDGeer}.}
 \end{enumerate}
\end{theorem}
\begin{remark}
By ``coincide over $\CC$,'' we mean that the operators coincide when $\Theta$ and $\Theta_{\pm}$ are restricted to connected components of $\mathcal{M}_U(\CC)$.
\end{remark}
\begin{proof}
The first and second equalities follow from Corollary \ref{coro:comparems}, together with the definitions $\Theta^r := \pi\circ \iota^\ast\circ\left(D^r|_{\mathcal{F}_r}\right)$ and $\Theta^r_\pm : = \pi_\pm\circ\iota^\ast\circ\left(D^r|_{\mathcal{F}_r}\right)$.  The final statement follows from Equation \eqref{equ:theyrethesame}, together with the definitions of the projections $\pi$ and $\pi_\pm$.  As noted above, over $\CC$, automorphic forms defined as global sections of a sheaf can be identified with automorphic forms defined as holomorphic functions on Hermitian symmetric spaces like in the first part of this paper, and the differential operators are also compatible with this identification. 
\end{proof}
\begin{remark}
The operators $\Theta$ and $\Theta_\pm$ can be applied not only to scalar-valued automorphic forms but also to vector-valued ones.  That is, it is straightforward to define them on vector bundles of automorphic forms instead of just line bundles of automorphic forms.  Since the first part of the paper only handled scalar-valued ones, though, we emphasize that case here.
\end{remark}

\begin{remark}
Computation of the Maass--Shimura operators in coordinates shows that if we replace $\mathcal{F}_r$ by a larger submodule $\mathcal{F}'$ of $\mathcal{F}$, then the images of $\pi\circ \iota^\ast\circ\left(D^r|_{\mathcal{F}'}\right)$ and $\pi_\pm\circ\iota^\ast\circ\left(D^r|_{\mathcal{F}'}\right)$
 consist of $C^\infty$-automorphic forms that are not in general holomorphic.  One can see this already in the simplest example, namely $\Sp_2\times \Sp_2\hookrightarrow \Sp_4$.  On the larger modules of holomorphic forms merely vanishing to order $r$ in the normal direction or the $\pm$-direction, the holomorphic operators from earlier in the paper coincide with the holomorphic projection of these $C^\infty$ Maass--Shimura operators. Similarly, the operators $\Theta^r$ and $\Theta^r_\pm$ do not extend to algebraic operators on a larger subsheaf of $\mathcal{F}.$
Although we shall not need it here, readers seeking an explicit treatment in terms of coordinates might consult \cite[Section 4]{hasv} and \cite[Sections 3.1.1 and 8.4]{Eischen2012}.  Those wishing to investigate holomorphic operators even further might consult \cite{MH86}.
\end{remark}

\addcontentsline{toc}{section}{References}
\bibliography{references}
\bibliographystyle{amsalpha}

\end{document}